\documentclass[12pt,leqno]{article}
\usepackage{amsmath, amsthm, amsfonts, amssymb, color}
\usepackage{mathrsfs}
\newtheorem{thm}{Theorem}[section]

\newtheorem{lem}[thm]{Lemma}
\newtheorem{prp}[thm]{Proposition}

\theoremstyle{definition}

\newcommand{\scr}[1]{\mathscr #1}
\definecolor{wco}{rgb}{0.5,0.2,0.3}

\numberwithin{equation}{section}

\date{}

\begin{document}

\newcommand{\lra}{\longrightarrow}
\newcommand{\ua}{\uparrow}
\newcommand{\da}{\downarrow}
\newcommand{\dsp}{\displaystyle}

\def\R{{\mathbb R}}
\def\I{I}
\def\paral{/\kern-0.55ex/} 
\def\parals_#1{/\kern-0.55ex/_{\!#1}} 
\def\n#1{|\kern-0.24em|\kern-0.24em|#1|\kern-0.24em|\kern-0.24em|} 
\def\ev{{\mathop {\mathrm ev}}}
\def\Ric{{\mathop {\rm Ric}}}
\def\Hess{{\mathop {\mathrm Hess}}}
\def\div{\mathop {{\mathrm \mathbf div}}}
\def\Cov{{ {\rm \bf   Cov}}}
\def\Var{{{\rm \bf Var}}}
\def \vGamma{{{\bf \Gamma }}}

\newcommand{\p}{{\bf P}}
 \def\D{\mathbf  D}
\def\E{\mathbf E}
\def\EE{{\scr E}}
\def\H{{\mathbf H}}
\def\Z{\mathbf Z}

 \def\ff{\frac}
  \def\ss{\sqrt}
  \def\d{\text{\rm{d}}} \def\loc{\text{\rm{loc}}}
\def\dd{\delta}
 \def\DD{\Delta}
 \def\vv{\varepsilon}

\def\<{\langle} \def\>{\rangle}

\newcommand{\Ex}{{\bf E}} 
\def\si{\sigma} \def\ess{\text{\rm{ess}}}
\def\beg{\begin} \def\beq{\beg}
  \def\F{{\mathcal F}}
\def\e{\text{\rm{e}}}
\def\Cut{\text{\rm Cut}}
\def\supp{\text{\rm supp}}

\noindent %

\title{\bf A Concrete Estimate For The Weak Poincar\'{e} Inequality On Loop Space}

\author{Xin Chen, Xue-Mei Li and Bo Wu}

\date{}
\maketitle

\ \  \ \ \ \ \ \ Mathematics Institute, University of Warwick, Coventry CV4 7AL, U.K.

\begin{abstract}
The aim of the paper is to study the pinned Wiener measure on the loop space over a simply connected compact Riemannian manifold together with the Hilbert space structure induced by Mallianvin calculus and
the induced Ornstein-Uhlenbeck operator $d^*d$. We give a concrete estimate for the weak Poincar\'{e} inequality for the O-U Dirichlet form on loop space over simply connected compact Riemannian manifold
with strict positive Ricci curvature.

\end{abstract}

\section{Introduction}

{\bf A.} Let $M$ be a  compact Riemannian manifold. For $a, b\in M$, we consider the pinned path space $\Omega_{a,b}$ over $M$,  
$$ \Omega_{a,b}=\{\omega \in C([0,1], M); \omega(0)=a,\ \omega(1)=b\}, $$
which is a smooth Finsler manifold with compatible distance function
$$d_{\infty}(\omega,\gamma):=\sup_{s \in [0,1]}d(\omega(s), \gamma(s)).$$
When $a=b$,  we have  the loop space over $M$, based at $a$.  
Let $\F C_b^{\infty}(\Omega_{a,b})$ be the collection of smooth cylinder function on $\Omega_{a,b}$. Each $F \in \F C_b^{\infty}(\Omega_{a,b})$ is determined by a smooth function $f$ on $M^{n}$
and a partition $0<s_1<\dots<s_n<1$ of $[0,1]$:
\begin{equation}\label{c0.1}
F(\omega)=(\ev_{s_1, \dots, s_n})_*f=f(\omega(s_1), \dots \omega(s_n)).
\end{equation}

For each $T>0$, endow $\Omega_{a,b}$
 with the pinned Wiener measure  $\p  _{a,b}^T$,  which is derived by  pushing forward the standard 
Brownnian bridge measure on the space of the pinned curves of  $C([0,T];M)$ with starting point $a$ and  
ending point $b$, to  $C([0,1];M)$ through the  rescaling map $\omega(t)\mapsto \omega(\frac{t}{T})$.  
The measure $\p  _{a,b}^T$ can be equally defined through its integration over  smooth cylindrical functions
of type (\ref{c0.1}):
\begin{equation*}
\begin{split}
&\int F(\omega) \p  _{a,b}^{T}(d\omega)\\
&=\frac {1}{p_{T}(a,b)}\int_{M^n} \; f(x_1,x_2,\dots x_n)\cdot
p_{s_1 T}(a,x_1)p_{(s_2-s_1)T}(x_1,x_2)\dots
p_{(1-s_n)T}(x_n,b) \prod_{i=1}^n \d x_i .
\end{split}
\end{equation*}
where $p_t(x,y)$ is the heat kernel on $M$. Write $\p  _{a,b}$ for $\p  _{a,b}^1$ for simplicity with corresponding expectation denoted by $\E  _{a,b}$.

{\bf B. } Let $\omega(s)$ be the canonical process on $\Omega_{a,b}$, $\F_s$ be the natural filtration and
$\F=\F_1$. Then
$\omega(s)$ is a semi-martingale with $(\Omega_{a,b},\F, \F_s, \p_{a,b}^T)$,
 see \cite{Dr3}. Denote by $\parals_{s,t}(\omega):\ T_{\omega(s)} M \rightarrow T_{\omega(t)}M$
 the stochastic paralell translation along the continuous path $\omega(\cdot)$, which is $\p  _{a,b}^T$ a.s defined. 
Write $\parals_s=\parals_{0,s}$.  Let $\H$ be  the standard Cameron-Martin space on an Euclidean space $\R^n$, $$\H=\left\{\sigma: [0,1]\to \R^n \; \big |   \int_0^1| \dot \sigma(s)|^2 ds<\infty\right\}.$$
We identify $T_aM$ with $\R^n$ and define the Bismut's tangent space $\H _{\omega}^0$ in $\Omega_{a,b}$:
\begin{equation*}
\H _{\omega}^0=\left\{\parals_\cdot(\omega) h_\cdot \; \big| \;   h\in H ,    h_0=0, h_1=0 \right\},
\end{equation*}
which is a Hilbert space with the inner product:
\begin{equation*}
(X,Y)_{\H _{\omega}^0}=\int_0^1
\Big\<\frac{d}{dt}(\parals_{t,0}(\omega)X_t)),
\frac{d}{dt}(\parals_{t,0}(\omega)Y_t))\Big\>_{T_a M}\d t
\end{equation*}
and corresponding norm $|-|_{\H_\omega^0}$.

Consider the differential operator  $d$ which sends a differentiable function on $\Omega_{a,b}$
(viewed as a Finsler manifold) to a differential 1-form. For $F$ a smooth cylindrical function,  $dF$ as a bounded linear map on $T\Omega_{a,b}$ can be considered as a bounded linear map on Bismut tangent space.
 By the Riesz representation theorem, there is the  $H^1$ gradient $\D _0F(\omega) \in \H _{\omega}^0$, given by  
 \begin{equation*}      
(\D _0F(\omega),X_h)_{\H _{\omega}^0}=d F(X_h),
\end{equation*}
for all vectors $X_h$ of the form $X_h(s)=\paral_s(\omega)h_s$ in $\H_\omega^0$.
In particular, for cylindrical function $F=(ev_{s_1, \dots, s_n})_*f$ of the form (\ref{c0.1}),
\begin{equation*}
\D _0F(\omega)(t)=\sum_{i=1}^{n}\parals_{s_i,t}(\omega)
\nabla_{i}f(\omega(s_1),\omega(s_2),\dots\omega(s_n))\cdot
G_0(s_i,t) 
\end{equation*}
where  $\nabla_{i}f\in T_{\omega(s_i)}M$ is the value at $\omega(s_i)$ of the gradient  of $f$ as a function of the $i$th variable at the point  $(\omega(s_1),\omega(s_2),\dots\omega(s_n))$ and $G_0(s,t)=s\wedge t-s\cdot t$, $0\leqslant s,t \leqslant 1$,  is the Green function  of the Gaussian measure on $\mathbb{R}^n$.  Also 
\begin{equation}\label{c0.2}
\begin{split}
|\D_0F|^2_{\H _{\omega}^0}
=\sum_{i,j=1}^n G_0(s_i,s_j)\cdot &\langle \parals_{s_i,s_j}(\omega)
\nabla_{i}f(\omega(s_1),\dots\omega(s_n)),
\nabla_{j}f(\omega(s_1),\dots\omega(s_n))
\rangle_{\omega(s_j)}
\end{split} 
\end{equation}

For each $T>0$, the  quadratic form  defined on smooth cylinder function by 
\begin{equation*}
\widetilde{\scr{E}}_{a,b}^{T}(F,F):=\int_{\Omega_{a,b}} |\D_0F|^2_{\H_\omega^0}
\p  _{a,b}^{T}(d\omega),  
\end{equation*}
can be extended to a Dirichlet form $\scr{E}_{a,b}^{T}$, which is due to an integration by parts formula, see \cite{Dr3}.
The domain of the Dirichlet form is $\scr{D}(\scr{E}_{a,b}^T)$ is the the same 
as the domain of the closure of the gradient
operator $\D_0$. Follow the custom, we call this  Dirichlet form the O-U Dirichlet form. And 
we denote $\scr{E}_{a,b}$ for $\scr{E}_{a,b}^1$ for simplicity. 

\medskip

If $\mu$ is a probability measure, we denote by $\E[F; \mu]$ the average of a 
function $F \in L^2(\mu)$ with respect to this measure and 
$\Var(F;\mu)=\E (F^2; \mu)-[\E(F; \mu)]^2$ the corresponding variance. The main theorem of the paper is:

\medskip

{\bf Theorem \ref {MainTheorem}. }  Let $M$ be a  simply connected compact manifold  with strict
positive Ricci curvature.  For any small $\alpha>0$, there exists a constant $s_0>0$ such
that the following weak Poincar\'{e}  inequality holds, i.e.
\begin{equation*}
\Var(F;\p  _{a,a})\leqslant
\frac{1}{s^\alpha}\EE_{a,a}(F,F) +
s||F||_{\infty}^2,\quad s\in(0,s_0),\ \ F\in \scr{D}(\scr{E}_{a,a}).
\end{equation*}
 And the constant $s_0$ does not depend on the starting point $a \in M$. 

\bigskip

 {\bf C. Historical Remark.}
 The Ornstein-Uhlenbeck   operator and the Ornstein-Uhlenbeck process plays an important role in the development of the $L^2$ theory on loop spaces, c.f. \cite{ELi3}.
  The study of the functional inequalities for O-U Dirichlet form with respect to  the Wiener measure (on path space) and to the pinned Wiener measure (on loop space) goes back a long way. For the Wiener measure on path space
over a compact manifold, it turns out that there is no fundamental topological or geometrical obstruction to the 
  validity of the Poincar\'{e} inequality. See e.g. the work of Fang  \cite{Fa} for the existence of a Poincar\'{e}  inequality for O-U Dirichlet form  and that of Hsu \cite{Hsu} for the  existence of a logarithmic Sobolev inequality. 

 But for the case of loop space over 
a compact manifold $M$, the problem seems much more complicated. Gross  \cite{Gr} pointed out that Logarithmic Sobolev inequality does not hold for O-U Dirichlet form when $M=S^1$ and he proved instead a Logarithmic Sobolev inequality  plus a potential term when $M$ was a compact Lie group.  In general the geometry and the topology of the manifold will play a significant role. In particular a Poincar\'{e} inequality does not hold for the Dirichlet form
with respect to pinned Winner measure if the underlying manifold is not simply connected, as the indicator function of each connected component of the loop space is in the domain of the O-U Dirichlet
form, see  Aida \cite{Ai3}. Furthermore, in \cite{EB}, Eberle constructed a simply connected compact Riemannian manifold on the loop space over which the Poincar\'{e}  inequality for O-U Dirichlet form did not hold. As transpired in his proof,  the validity of the Poincar\'{e}  
inequality may depend on the starting point of the based loop space. A Clark-Ocone formula with a potential was deduced by  Gong and Ma  \cite{GM}, which led to their discovery of a  Logarithmic Sobolev inequality with a potential
on loop space over general compact manifold. See also Aida \cite{Ai1}. In their results, the simply connected condition is not needed for the underlying manifold. Aida \cite{Ai4}, on the other hand, deduced a Clark-Ocone formula which led to a Logarithmic Sobolev inequality for a modified Dirichlet form,  under suitable conditions on the small time asymptotics of the Hessian of the logarithm of the heat kernel of the underlying manifold. Built on that, a  Poincar\'{e}  inequality is shown to hold for the O-U Dirichlet form on the loop space over hyperbolic space, see  Chen-Li-Wu \cite{CLW}.  

Another development in the positive direction comes from Eberle \cite{Ai2},  where it was shown that a local Poincar\'{e} ineqaulity  hold for the O-U Dirichlet form on loop space over compact manifold. A parallel result was given by Aida  \cite{Ai2}: when $M$ was simply connected, the O-U Dirichlet form had  the weak spectral gap property. By the weak spectral gap property for a
Dirichlet form $\scr{E}$ in $L^2(\p  )$ it is meant  that $F_n\rightarrow 0$ in probability for any sequence of functions 
$\{F_n\}_{n=1}^{\infty}\subset \scr{D}(\scr{E})$ satisfying the following 
conditions, 
\begin{equation*}
\sup_n ||F_n||_{L^2}\leqslant 1, \quad \ E(F)=0, \quad \lim_{n\rightarrow \infty}\scr{E}(F,F)=0, 
\end{equation*}
 see also Kusuoka \cite{Ku}. Although we do not know the relation between Eberle's local Poincar\'{e} ineqaulity and Aida's weak spectral gap property, it was noted in R\"{o}ckner and Wang in \cite{Wang} , the weak spectral gap property  was equivalent to the following weak Poincar\'{e}  inequality: 
 \begin{equation*}
\Var_{\mu}(f) \leqslant \beta(s)\scr{E}(f,f) +s ||f||_{\infty}^2, 
\ s\in (0,s_0)\ \ \ f\in \scr{D}(\scr{E})\cap L^{\infty}(\mu) \end{equation*}
Here $\beta:\R^+\rightarrow \R^+$ is a non-increasing function and
$s_0>0$ is a constant. 
And in \cite{Ai5}, Aida used such 
weak Poincar\'{e} inequality to give an estimate on the spectral gap of a Schr\"{o}dinger operator on the loop 
space. We refer the reader to Wang \cite{W2} for  analysis, development and historical references on such inequalities. 
Our contribution here is the  concrete estimate of $\beta(s)$ in the  inequality above. The main difficulty here is to find suitable exhausting local sets  replacing the role played by geodesic balls in the proof
of weak Poincar\'{e} inequality on finite dimensional manifolds (see \cite{W2}).  
The local sets Eberle taking in \cite{EB1} are not suitable for 
our proof. So in our approach, we use a different collection of  local sets. On such 
local sets, we do not derive the exact local Poincar\'{e} inequality, some additional term of 
the $L^{\infty}$ norm will appear in the estimate, but finally we can control such terms 
to get a global weak Poincar\'{e} inequality.

The  paper is organised as follows. In Section 2, we introduce notation and state some results, especially that of Eberle \cite{EB1, EB2} on which our proof is based on. In Section 3, we give some variance estimate for small time.  In Section 4,
A weak Poincar\'{e} type inequality  for the distribution of  the Brownian bridge evaluated at $N$ equal time intervals is given. We use  a combination of small time asymptotics and Poincar\'{e} inequality for the Wiener measure to control the growth of the constants with $N$. In paritcular, some of the 
methods in this section are inspired by \cite{EB2} and \cite{HKS}. In section 5, the main theorem is proved by reducing the variance of a function on the loop space to the variance of a function on a product manifold which is localized to subsets which are chains of small geodesic balls, and the variance of functions  on some sub-path with respect to pinned Wiener measure with small time parameter.  

\bigskip

\noindent
{\bf{Acknowledgment.}}  We would thank Feng-Yu Wang for inspiring discussions and Courant Institute for its hospitality 
during the completion of the work. This research is supported by  the EPSRC (EP/E058124/1).


\section{Notations and known results}

Let $\{B_s\}$ be the $T_a M$ valued stochastic anti-development of the canonical process
$\omega(s)$, which is a semi-martingale with $(\Omega_{a,b},\F, \F_s, \p_{a,b}^T)$. It is however not a 
Brownian motion, see \cite{Dr3}.  Denote by $L(\R^n; T_aM)$ the set of all linear maps 
from $\R^n$ to $T_aM$. 

\begin{lem} [\cite{EB1}]  \label{l1}
Let $\{A_s(\omega), \omega\in \Omega_{a,b}, 0\leqslant s\leqslant 1\}$ be a $L(\R^n; T_aM)$ valued adapted process  such that
$s\to A_s(\omega)$ is $C^1$ for every $\omega$ and 
$$\sup_{\omega \in \Omega_{a,b}}\sup_{s \in [0,1]}|A'_s(\omega)|< \infty. $$
Suppose that $H_s(\omega)=A_s(\omega)h_s$ for  some $h\in \H$ with $h_1=0$ and 
 $X_\cdot(\omega)=\parals_\cdot(\omega)A_\cdot(\omega)h_\cdot$. 
Define
 $$\delta^T _uX:=\int_0^u\big[ T^{-1}H'_s+ \frac{1}{2}\parals_s^{-1}\Ric^{\#}_{\omega(s)}(\parals_s H_s)\big]dB_s,\ \ 0\leqslant u<1$$
Then  $$\delta^T X:=\lim_{u\to 1} \delta_u^T X$$ exists in $L^1(\Omega_{a,b};\p  _{a,b}^T)$  and the limit 
is in $L^2(\Omega_{a,b};\p  _{a,b}^T)$. 
\end{lem}
If $a,b \in M$ are not in the cut locus of each other, we take  $A_{\cdot}$ such that 
$\parals_{\cdot}A_{\cdot}$ is the damped stochastic parallel transport and take $H_{\cdot}$ to be parallel push back of the Jacobi fields 
along the unique geodesic connecting $a$ and $b$ with initial vector $v$. By a result of Malliavin-Stroock, the variance of $\delta^T X$ defined
in above lemma with respect to $P^T_{a,b}$ are uniformly bounded for $a,b, v, T$ in compact sets. In fact, we have the following lemma, 
\begin{lem} [ \cite{EB1}, \cite{MS}]
Let $b \in M \setminus \text{Cut}(a)$, $\ v\in T_aM$, and $T>0$. There is a vector 
$X_s^{T,a,b,v}=\paral_s H_s$, with initial value $X_0^{T,a,b,v}=v$  for $H$ as in Lemma 2.1, such that for every $T_0 >0$ and 
$r\in (0,\text{inj}_M)$, 
 \begin{equation*}
\sup_{T \in (0,T_0]}\rho(T,r)<\infty,
\end{equation*}
where
\begin{equation}\label{d1.10}
\rho(T,r):=\sup_{a,b\in M, d(a,b)\leqslant r,\ v \in T_aM, |v|=1}\big\{\Var\big(\delta^T X^{T,a,b,v}; \p  _{a,b}^T \big)\big\}.
\end{equation}
\end{lem}

The next lemma deals with 
the derivative with  starting point of the expectation under pinned Wiener measure,

\begin{lem}[\cite{EB1}]
Let $v \in T_a M$. For each $X_s=\parals_s(\omega)H_s(\omega)$ with 
$H_s(\omega)$ as in  lemma 2.1, and 
that $X_0=v\ \p  ^T_{a,b}$ a.s,
\begin{equation}
d_a\big(\E  ^T_{.,b}[F]\big)[v]=\E^T_{a,b}[dF(X)]-
\Cov\big(\delta^T X,F;\p^T_{a,b}\big) 
\end{equation}
for all smooth cylinder function $F\in\F C_b^\infty(\Omega_{a,b})$.
\end{lem}
 
For two paths $\omega_1, \omega_2$ with
$\omega_1(1)=\omega_2(0)$, define $\omega_1\vee\omega_2$ as
following:
\begin{equation*}
\omega_1\vee\omega_2(s)=
\begin{cases}
\omega_1(2s) \ &\text{if}\ s\in [0,1/2],\\
\omega_2(2s-1)\ &\text{if}\ s\in [1/2,1].
\end{cases}
 \end{equation*}
For each $\omega$ in $\Omega_{a,b}$, we can find one and 
only one pair of $\widetilde{\omega_1}$, $\widetilde{\omega_2}$ to satisfy that
$\omega=\widetilde{\omega_1}\vee\widetilde{\omega_2}$, 
then for each fixed $T>0$, $\omega  \in \Omega_{a,b}$ with
$a,b \notin \Cut(\omega(1/2))$ and 
$\omega=\widetilde{\omega_1}\vee\widetilde{\omega_2}$,   
$v \in
T_{\omega(1/2)}M$, let

\begin{equation*}
\widehat{X}_s^{T,v}(\omega)=
\begin{cases}
X_{1-2s}^{T/2,\omega(1/2),\omega(0),v}(\widetilde{\omega_1}^{-1}) \ &\text{if}\ s\in [0,1/2],\\
X_{2s-1}^{T/2,\omega(1/2),\omega(1),v}(\widetilde{\omega_2})\ &\text{if}\
s\in [1/2,1]
\end{cases}
\end{equation*}
where $X_s^{T,a,b,v}$ is as  in lemma 2.2 and $\widetilde{\omega_1}^{-1}(s):=\widetilde{\omega_1}(1-s), 
\ 0\leqslant s \leqslant 1$ is the time inverse of the path $\widetilde{\omega_1}$.

For $F\in\F C_b^\infty(\Omega_{a,b})$ and $\omega \in \Omega_{a,b}$, let
\begin{equation}\label{d1.21}
\vGamma^T(F)(\omega)=\begin{cases}\sup_{\{v\in T_{\omega(1/2)}M, |v|=1\}} \;(dF(\widehat{X}_s^{T,v}))^2(\omega), \quad  &\hbox{if } a,b\notin \Cut(\omega(1/2)) 
\\
0, \quad & \hbox{otherwise} .
\end{cases}
\end{equation}
For each smooth cylinder function $F\in\F C_b^\infty(\Omega_{a,b})$, there exists a unique function $\widetilde{F}(\widetilde{\omega_1},\widetilde{\omega_2})$
, defined on $\bigcup_{z \in M}\Omega_{a,z}\times \Omega_{z,b}$, such that $\widetilde{F}(\widetilde{\omega_1},\widetilde{\omega_2})=
F(\widetilde{\omega_1}\vee\widetilde{\omega_2})=F(\omega)$
for each $\omega$ in $\Omega_{a,b}$ with $\omega=\widetilde{\omega_1} \vee \widetilde{\omega_2}$.



\begin{lem}[ \cite{EB1}]\label{L1} For $a,b\in M$, $T>0$, and $r>0$, 
denote $U_{a,b}^r=B_r(a)\cap B_r(b)$ ($B_r(a)$ means the ball with center $a$ and radius
$r$) and
\begin{equation}\label{f2}
\mu_{a,b}^T(dx)=\frac{p_{T/2}(a,x)p_{T/2}(x,b)}{p_T(a,b)}dx.
\end{equation}

\begin{itemize}
\item There exists a positive number $R_1$, such that when $r \in (0,R_1)$,
\begin{equation}\label{lemma-variance-1}
\begin{split}
\Var(F;\p  _{a,b}^T)&\leqslant 2q(T,r)\E  _{a,b}^T[\vGamma   ^T(F)]\\
&+(1+4q(T,r)\rho(T/2,r))\int_{U_{a,b}^r}\big\{\E  _{a,x}^{T/2,1}
[\Var_2(\widetilde{F}(\widetilde{\omega_1},\widetilde{\omega_2});\p  _{x,b}^{T/2})]\\
&+\E  _{x,b}^{T/2,2}
[\Var_1(\widetilde{F}(\widetilde{\omega_1},\widetilde{\omega_2});\p  _{a,x}^{T/2})]\big\}\mu^T_{a,b}(dx) 
\end{split}
\end{equation}
for every smooth cylinder function  $F:\Omega_{a,b}\mapsto \mathbb{R}$ such that 
$F(\omega)=0$ if $\omega(1/2)$ is not in 
$U_{a,b}^r$.  Here  $\vGamma   ^T(F)$, $\rho(T/2,r)$ are defined by
 (\ref{d1.21}) and (\ref{d1.10}) respectively. $\E  ^{T/2, i}$, $\Var_{i}$ indicates  that the corresponding expectation or 
variance is taken with respect to the ith-subpath $\widetilde{\omega}_i$, $i=1,2$, 
\item The constant $q(T,r)$ in above inequality does not depend on $a,b \in M$ and satisfies 
\begin{equation}\label{f8}
\varlimsup_{T\downarrow 0}T^{-1}q(T,r)\leqslant \frac{1+Kr^2}{4}\ \ \forall r \in (0,R_1). 
\end{equation}
for some $K>0$.
\end{itemize}
\end{lem}



\section{Some estimate on the variance with small time paremeter}
The following lemma gives a short time asymptotics of the variance. It is crucial for the proof of main result  in this section. For simplicity, 
in the remaining part of the paper, the constant $C$ will change
according to different situation but we will clarify which parameter
such $C$ depends on. 
At first, we state a lemma deriving from lemma 2.4 by some cut-off procedures, 
\begin{lem}
There exists a number $R_1>0$ such that for all $a, b\in M$  with $d(a,b)<r<R_1$, the following holds  for any small number $\eta>0$ and smooth cylindrical function $F$ on $(\Omega_{a,b},P^T_{a,b})$, as soon as $T<T_1(\eta,r)$ for 
a positive number $T_1(\eta,r)$  depending  only on $\eta$ and $r$.
\begin{equation}
\begin{split}\label{l1.6}
&\Var(F;\p  _{a,b}^T)\\
&\leqslant 4q(T,r)\E  _{a,b}^T
\big[\vGamma   ^T(F)I_{\{\omega(1/2) \in U_{a,b}^r \}}\big]\\
&+\big(1+4q(T,r)\rho(T/2,r)\big)\int_{U_{a,b}^r}\Big\{\E  _{a,x}^{T/2,1}
[\Var_2(\widetilde{F}(\widetilde{\omega_1},\widetilde{\omega_2});\p  _{x,b}^{T/2})]\\
&+\E  _{x,b}^{T/2,2}
[\Var_1(\widetilde{F}(\widetilde{\omega_1},\widetilde{\omega_2});\p  _{a,x}^{T/2})]\Big\}\mu^T_{a,b}(dx)
+\Big(6+\frac{128q(T,r)}{\eta^2 r^2}\Big)\text{e}^{-\frac{(1-4\eta)r^2}{2T}}||F||_{\infty}^2.
\end{split}
\end{equation}
Here the measure $\mu_{a,b}^T(dx)$ is the distribution of the mid-point of the Brownian Bridge, given by (\ref{f2}), $U_{a,b}^r:=B_r(a)\cap B_r(b)$, $q(T,r)$ is some constant satisfying
(\ref{lemma-variance-1}), and  $\E  ^{T/2, i}$, $\Var_{i}$ indicates  that the expectation or 
variance is taken with respect to the subpath $\widetilde{\omega}_i$.  
\end{lem}

\noindent 
{\it Remark: } The constants $R_1$ is the same as that in lemma 2.4. It is smaller than the injectivity radius of $M$.

\begin{proof}  {\bf Step 1.}  For a positive $r$ smaller than the injectivity radius of $M$,  $a,b \in M$ with $d(a,b)<r$, define a function $\varPsi_{a,b}:  \Omega_{a,b}\to \R$ by
\begin{equation}\label{f1.0}
\varPsi_{a,b}(\omega):=\varphi(d(a,\omega(1/2)))\cdot\varphi(d(b,\omega(1/2))).
\end{equation}
Here $\phi$ is a a smooth function 
$\varphi:\mathbb{R}^{+}\rightarrow \mathbb{R} $ satisfying
\begin{equation}\label{e2}
\varphi(s)=
\begin{cases}
1,\ &\text{if}\ s\leqslant (1-\eta)r,\\
0,\ &\text{if}\ s\geqslant r
\end{cases} \quad 
\ \text{and}\ |\varphi'|\leqslant \frac{2}{\eta r}.\end{equation}
 Then the function $\varPsi_{a,b}$ is in $\scr{D}(\scr{E}_{a,b}^{T})$, the domain
 of the O-U Dirichlet form, and 
$\sup_{\omega \in \Omega_{a,b}} |\D _0\varPsi_{a,b}(\omega)|_{\H _{\omega}^0}
\leqslant \frac{4}{\eta r}$.  Furthermore we  show below that for all small $\eta>0$ there is constant $T_1(\eta, r)$ such that if $T<T(\eta, r)$,
\begin{equation}\label{l1.2}
\p  _{a,b}^T(\varPsi_{a,b}\neq 1)\leqslant
2\text{e}^{-\frac{(1-4\eta)r^2}{2T}}.
\end{equation}
We begin with estimating  the probability $$\p  _{a,b}^T\Big(d(a,\omega(1/2))>(1-\eta)r\Big)=\mu_{a,b}^T(B_{(1-\eta)r}(a)).$$
By Varadhan's estimate \cite{Va}, $$\text{lim}_{T\downarrow 0}T
\text{log}p_T(a,b)=-\frac{d^2(a,b)}{2}\quad \hbox{uniformly on}~
M\times M.$$ 
Hence for any $\eta>0$ small, there exists a constant $T_1(\eta,r)>0$, such that for
every $0<T<T_1(\eta,r)$,
\begin{equation}\label{c1.0}
-\frac{d^2(a,b)}{2T}-\frac{\eta^2 r^2}{2T}\leqslant
\text{log}p_T(a,b)\leqslant -\frac{d^2(a,b)}{2T}
+\frac{\eta^2 r^2}{2T}.
\end{equation}
 In the calculations that follows we assume that $0<T<T_1(\eta,r)$.  Note that  $d(a,b)<r$, 
\begin{equation*}
\begin{split}
&\p  _{a,b}^T\Big(d(a,\omega(1/2))>(1-\eta)r\Big)
\\
&=\frac{1}{p_T(a,b)}\int_{\{d(a,x)>(1-\eta)r\}}p_{T/2}(a,x)p_{T/2}(x,b)dx\\
&\leqslant
\text{e}^{\frac{d^2(a,b)}{2T}+\frac{\eta^2 r^2}{2T}}\int_{\{d(a,x)>(1-\eta)r\}}
\text{e}^{-\frac{d^2(a,x)}{T}
+\frac{\eta^2 r^2}{T}}p_{T/2}(x,b)dx\\
&\leqslant
\text{e}^{\frac{r^2}{2T}+\frac{\eta^2 r^2}{2T}}\text{e}^{-\frac{(1-\eta)^2 r^2}{T}+\frac{\eta^2 r^2}{T}}
 \leqslant
\text{e}^{-\frac{(1-4\eta)r^2}{2T}}.
\end{split}
\end{equation*}
Similarly, 
$$\p  _{a,b}^T\big(\; d(b,\omega(1/2))>(1-\eta)r \;\big)\leqslant\text{e}^{-\frac{(1-4\eta)r^2}{2T}}.$$ 
Hence
\begin{equation*}
\begin{split}
\p  _{a,b}^T(\varPsi_{a,b}\neq 1)
&\leqslant
\p  _{a,b}^T \left (d(a,\omega(1/2))>(1-\eta)r \right)+\p  _{a,b}^T
\left( d(b,\omega(1/2))>(1-\eta) \right)\\
&\leqslant
2\text{e}^{-\frac{(1-4\eta)r^2}{2T}}.
\end{split}
\end{equation*}

\noindent {\bf Step (b).}  Let $R_1$ be the constant  in Lemma 2.4. Assume that $r<R_1$ and 
  we first  observe that
\begin{equation}
\begin{split}\label{l1.5}
&\Var(F;\p  _{a,b}^T)=\E  _{a,b}^T F^2-(\E  _{a,b}^TF)^2\\
&\leqslant \E  _{a,b}^T (F\varPsi_{a,b})^2+
||F||_{\infty}^2\p  _{a,b}^T(\varPsi_{a,b}\neq 1)-(\E  _{a,b}^T[F-F\varPsi_{a,b}+F\varPsi_{a,b}])^2\\
&\leqslant
\Var(F\varPsi_{a,b};\p  _{a,b}^T)+3||F||_{\infty}^2\p  _{a,b}^T(\varPsi_{a,b}\neq
1)\\
&\leqslant\Var(F\varPsi_{a,b};\p  _{a,b}^T) + 
6\text{e}^{-\frac{(1-4\eta)r^2}{2T}}||F||_{\infty}^2
\end{split}
\end{equation}
Since $\varPsi_{a,b}(\omega)=0$ when 
$\omega(1/2)\notin U^r_{a,b}$,  Lemma 2.4 applies to $F\Psi_{a,b}$ and we have, 
\begin{equation*}
\begin{split}
&\Var(F\varPsi_{a,b};\p  _{a,b}^T)
\leqslant 2q(T,r)\E  _{a,b}^T[\vGamma   ^T(F\varPsi_{a,b})]\\
&+\big(1+4q(T,r)\rho(T/2,r)\big)\int_{U_{a,b}^r}\Big\{\E  _{a,x}^{T/2,1}
[\Var_2(\widetilde{F\varPsi_{a,b}}(\widetilde{\omega_1},\widetilde{\omega_2});\p  _{x,b}^{T/2})]\\
&+\E  _{x,b}^{T/2,2}
[\Var_1(\widetilde{F\varPsi_{a,b}}(\widetilde{\omega_1},\widetilde{\omega_2});
\p  _{a,x}^{T/2})]\Big\}\mu^T_{a,b}(dx),
\end{split}
\end{equation*}
We next deal with the terms $\Var_i(\widetilde{F\varPsi_{a,b}}(\widetilde{\omega_1},\widetilde{\omega_2});\p  _{x,b}^{T/2})$.
Since $\varPsi_{a,b}(\omega)=\varphi(d(a,\omega(1/2)))\cdot\varphi(d(b,\omega(1/2)))$ 
is determined by $\omega(1/2)$, for $i=1,2$.
\begin{equation*}
\begin{split}
\Var_i(\widetilde{F\varPsi_{a,b}}(\widetilde{\omega_1},\widetilde{\omega_2});\p  _{x,b}^{T/2})
&=\varphi(d(a,x))^2\cdot \varphi(d(x,b))^2\cdot\Var_i(\widetilde{F}(\widetilde{\omega_1},\widetilde{\omega_2});\p  _{x,b}^{T/2})\\
&\leqslant \Var_i(\widetilde{F}(\widetilde{\omega_1},\widetilde{\omega_2});\p  _{x,b}^{T/2})
I_{\{x \in U_{a,b}^r\}}.
\end{split}
\end{equation*}
Consequently
\begin{equation}
\begin{split}\label{l1.3}
\Var(F\varPsi_{a,b};\p  _{a,b}^T) 
&\leqslant 2q(T,r)\E  _{a,b}^T[\vGamma   ^T(F\varPsi_{a,b})]\\
&+\big(1+4q(T,r)\rho(T/2,r)\big)\int_{U_{a,b}^r}\Big\{\E  _{a,x}^{T/2,1}
[\Var_2(\widetilde{F}(\widetilde{\omega_1},\widetilde{\omega_2});\p  _{x,b}^{T/2})
I_{\{x \in U_{a,b}^r\}}]\\
&+\E  _{x,b}^{T/2,2}
[\Var_1(\widetilde{F}(\widetilde{\omega_1},\widetilde{\omega_2});\p  _{x,b}^{T/2})
I_{\{x \in U_{a,b}^r\}}]\Big\}\mu^T_{a,b}(dx),
\end{split}
\end{equation}
Let $S_a M:=\{v \in T_a M,\ |v|=1\}$. For each $\omega \in \Omega_{a,b}$, by the
definition of $\vGamma   ^T$ in (\ref{d1.21}), 
\begin{equation}
\begin{split}\label{l1.4}
&\vGamma   ^T(F\varPsi_{a,b})(\omega)=\sup\big\{[d(F\varPsi_{a,b})(\widehat{X}^{T,v})]^2
(\omega); v \in S_{\omega(1/2)}M\big\}\\
&\leqslant 2\sup\big\{[dF(\widehat{X}^{T,v})]^2\varPsi_{a,b}^2;v \in
S_{\omega(1/2)}M\big\}+
2\sup\big\{F^2[d\varPsi_{a,b}(\widehat{X}^{T,v})]^2(\omega);v \in S_{\omega(1/2)}M\big\}\\
&\leqslant 2\vGamma   ^T(F)(\omega)I_{\{\omega(1/2) \in U_{a,b}^r \}}+2||F||_{\infty}^2
\sup\big\{\<v,\nabla_x [\varphi(d(a,x))\cdot \varphi(d(x,b))]\>^2;
v \in S_{x}M\big\}\\
&\leqslant 2\vGamma   ^T(F)(\omega)I_{\{\omega(1/2) \in U_{a,b}^r \}}
+ \frac{32}{\eta^2 r^2}||F||_{\infty}^2I_{\{\varPsi_{a,b}(\omega)\neq
1\}}.
\end{split}
\end{equation}
The required inequality  (\ref{l1.6}) follows from (\ref{l1.5}), (\ref{l1.3}) and (\ref{l1.4}).

\end{proof}

\begin{prp}\label{l1.1}
There is a constant $R_0$ such that for each small $\eta>0$ 
the following holds on $(\Omega_{a,b}, \p_{a,b}^T)$ provided that $d(a,b)<r<R_0$
and  $0<T<T_0(\eta,r)$ for some $T_0(\eta,r)>0$:
\begin{equation*}
\Var(F;\p  _{a,b}^T)\leqslant TC(r)\E  _{a,b}^T
(|\D _0 F|^2_{\H _{\omega}^0})+
C(\eta,r)||F||_{\infty}^2 e^{-\frac{(1-4\eta)r^2}{2T}},\quad F\in\F C_b^\infty(\Omega_{a,b})
\end{equation*}
Here $C(\eta,r)$, $C(r)$ are independent of $T$.
\end{prp}

\begin{proof} 
 By approximation, it suffices to show the 
inequality holds for all smooth cylindrical functions on dyadic partitions, e.g. of the form
\begin{equation}\label{e1}
F(\omega)=f\Big(\omega\Big(\frac{1}{2^m}\Big),
\omega\Big(\frac{2}{2^m}\Big)
,\dots\omega\Big(\frac{2^m-1}{2^m}\Big)\Big),\ f \in
C^{\infty}(M^{2^m}), \  m\in {\mathbb N^{+}}.
\end{equation}
For any $\omega\in \Omega_{a,b}$, let
\begin{equation*}
\omega_i(s)=\omega\Big(\frac{i-1+s}{2^k}\Big), \ \ s\in [0,1],
\ k\in \mathbb{N}^{+},\ 1\leqslant i \leqslant 2^k
\end{equation*}
For simplicity, we did not reflect the  index $k$ in the definition of the new path $\omega_i$.
For each smooth cylinder function $F$ and positive 
integer $k$
we define a unique function $F^{[k]}$ of $2^k$  sub-paths. It is defined on
$\bigcup_{z=(x_1,\dots x_{2^k-1})\in M^{2^k-1}}\prod_{i=1}^{2^k}\Omega_{x_{i-1},x_i}$
(here $x_0=a$ and $x_{2^k}=b$), such that for each $\omega \in \Omega_{a,b}$
$$F^{[k]}(\omega_1,\dots\omega_{2^k})=F(\omega).$$
In fact,  
$\int F^{[k]}(\omega_1,\dots\omega_{2^k}) \prod_{i=1}^{2^k} 
\p  _{x_{i-1},x_i}^{T/2^k}(d\omega_i)$
is a 
smooth version of the conditional expectation $\E  _{a,b}^{T}[F
|\omega(1/2^k)=x_1,\dots\omega(1-1/2^k)=x_{2^k-1}]$
and 
$F^{[1]}$ is the same as $\widetilde{F}$  in Lemma 2.4 and Lemma 3.1.  

For $N\geqslant 1$ and $T>0$ we define
the probability measure
$\mu_{a,b}^{N,T}$ in $M^{N-1}$ as,  
\begin{equation}\label{d1.2}
\mu_{a,b}^{N,T}(\d x):=
\frac{p_{\frac{T}{N}}(b,x_{N-1})p_{\frac{T}{N}}(x_{N-1},x_{N-2})\dots
p_{\frac{T}{N}}(x_1,a)}{p_T(a,b)}\d x_{N-1}\dots\d x_1.
\end{equation}

Fix a number $0<r<R_1$ for $R_1$ as  in Lemma 3.1,  $\eta>0$ and a positive number $T<T_1(\eta,r)$.  For the variance terms for $\widetilde{F}$ as a function of 
any of the two subpaths $\widetilde{\omega_1}, \widetilde{\omega_2}$ on the right 
side of inequality (\ref{l1.6}), we  apply  (\ref{l1.6}) from lemma 3.1, on each sub-path while keeping the other fixed, to obtain an  estimate on the variance of $F$ in terms of 
 the variances and the operation $\vGamma ^{T/2}$ for $F^{[2]}$ as a function of any of the four subpaths (note that $x \in U_{a,b}^r$, so 
we can use lemma 3.1 here).  Repeat with this procedure by mid-dividing the path and applying  (\ref{l1.6}).
The variance terms will finally vanish  after  a repetition of $m$ times for the smooth cylinder function of  type (\ref{e1}), and 
we have,

\begin{equation}
\begin{split}\label{c1.7}
&\Var(F;\p  _{a,b}^T)
\leqslant 4\sum_{k=0}^{m-1}G(k,T,r)q(T/2^k,r)\times\\
&\Big(
\sum_{j=1}^{2^k}\int_{U_j}\bigg\{\E  _{x_{j-1},x_j}^{T/2^k,j}
\big[\vGamma   ^{T/2^k,j}(F^{[k]})(\omega_1,\dots \omega_{2^k})
I_{\{\omega_j(1/2) \in U_{x_{j-1},x_j}^r \}}\big]\\
&\times\prod_{i\neq j}\p  _{x_{i-1},x_i}^{T/2^k}(d\omega_i)\bigg\}
\mu_{a,b}^{2^k,T}(dx)\Big)\\
&
+\sum_{k=0}^{m-1}G(k,T,r)\Big(6+\frac{128q(T/2^k,r)}{\eta^2 r^2}\Big)
2^k\text{e}^{-\frac{2^{k}(1-4\eta)r^2}{2T}}||F||_{\infty}^2
\end{split}
\end{equation}
where $G(0,T,r)=1$, 
$G(k,T,r)=\prod_{i=1}^k(1+4q(T/2^{i-1},r)\rho(T/2^i,r))$ for
each $k>0$,   
$U_j=\{x=(x_1, \dots, x_{2^k-1})\in M^{{2^k-1}} : d(x_{j-1}, x_j)<r\}$ for $j=1,2 \dots, 2^k$ ($x_0=a$ and $x_{2^k}=b$). We denote by
$\E  ^{T/2^k,j}_{x_{j-1},x_j}$ and   $\vGamma   ^{T/2^k,j}(F^{[k]})$
 taking the corresponding expectation and the operation $\vGamma   ^{T/2^k}$ (defined in (\ref{d1.21})) with respect to the $j$th sub-path for 
function $F^{[k]}$.

By (5.8) in the proof of lemma 5.1 in \cite{EB1}  if  $T$ is small enough, $$\sup_{k \in \mathbb{N}} G(k,T,r)<C(r).$$
By this and Lemma 3.3 below we can find a positive  number
$R_0<R_1$, such that for each $0<r<R_0$, there is a $T_2(r)>0$, 
when  $T<T_2(r)$ the following holds for all  positive integer $m$:
\begin{equation}
\begin{split}\label{c1.8}
&\sum_{k=0}^{m-1}G(k,T,r)q(T/2^k,r)\cdot
\Big(\sum_{j=1}^{2^k}\int_{U_j}\bigg\{\E  _{x_{j-1},x_j}^{T/2^k,j}
\big[\vGamma   ^{T/2^k,j}(F^{[k]})(\omega_1,\dots \omega_{2^k})
I_{\{\omega_j(1/2)\in U_{x_{j-1},x_j}^r\}}\big]\\
&\times\prod_{i\neq j}\p  _{x_{i-1},x_i}^{T/2^k}(d\omega_i)\bigg\}
\mu_{a,b}^{2^k,T}(dx)\Big)\\
&\leqslant TC(r)\E  _{a,b}^T|\D _0F|_{\H _{\omega}^0}^2
\end{split}
\end{equation}
Note that by part 2 of  Lemma 2.4 there is $T_0(\eta,r) <\min (T_2(r), T(\eta, r))$ such that if  $T<T_0(\eta,r)$, then 
$|q(T,r)|\leqslant C(r)$ for some constant $C$ depending only on $r$. Using this bound and the bound
on $\sup_k G(k,T,r)$ we see that for $T<T_0$, 
$$\sup_{k\in \mathbb{N}}G(k,T,r)\left(6+\frac{128q(T/2^k,r)}{\eta^2 r^2}\right)\leqslant C(r,\eta)$$
and 
\begin{equation}
\begin{split}\label{c1.9}
\sum_{k=0}^{m-1}G(k,T,r)\Big(6+\frac{128q(T/2^k,r)}{\eta^2 r^2}\Big)
2^k\text{e}^{-\frac{2^{k}(1-4\eta)r^2}{2T}} &\leqslant
C(r,\eta)\sum_{k=0}^{\infty}2^k\text{e}^{-\frac{2^k(1-4\eta)r^2}{2T}}\\
&\leqslant
C(r,\eta)\text{e}^{-\frac{(1-4\eta)r^2}{2T}}
\end{split}
\end{equation}
We conclude the proof from  (\ref{c1.7}),  (\ref{c1.8}) and  (\ref{c1.9}).
\end{proof}

\begin{lem} Let
$U_j=\{x=(x_1, \dots, x_{2^k-1})\in M^{{2^k-1}}: d(x_{j-1}, x_j)<r,\}$ for 
$j=1,2\dots 2^k$($x_0=a$ and $x_{2^k}=b$). 
We can find a $R_2>0$, for each $0<r<R_2$, there is a $T(r)>0$, 
when $T<T(r)$, we have
\begin{equation*}
\begin{split}\label{p1.1}
&\sum_{k=0}^{m-1}q(T/2^k,r)\cdot \Big(\sum_{j=1}^{2^k}
\int_{U_j}\bigg\{\E  _{x_{j-1},x_j}^{T/2^k,j}
\big[\vGamma   ^{T/2^k,j}(F^{[k]})(\omega_1,\dots \omega_{2^k})
I_{\{\omega_j(1/2)\in U_{x_{j-1},x_j}^r\}}\big]\\
&\times \prod_{i\neq j}\p  _{x_{i-1},x_i}^{T/2^k}(d\omega_i)\bigg\}
\mu_{a,b}^{2^k,T}(dx)\Big)\\
&\leqslant TC(r)\E  _{a,b}^T|\D _0F|_{\H _{\omega}^0}^2
\end{split}
\end{equation*}
Here $C(r)$ is independant with $T$.
\end{lem}
\begin{proof}
Following the notation from \cite{EB1}, let $\{h_{k,j};\ k\geqslant 0, 1\leqslant j \leqslant 2^k\}$
be the orthonormal basis of $H_0^{1,2}([0,1];\mathbb{R})$ consisting of
Schauder functions, i.e. $h_{0,1}(s)=s\wedge(1-s)$,
\begin{equation*}
\begin{cases}
h_{k,j}(s)=2^{-k/2}h_{0,1}(2^k s-(j-1)) \ &\text{if}\ s\in [(j-1)2^{-k},j2^{-k}],\\
h_{k,j}(s)=0\ &\text{otherwise}\  
\end{cases}
\end{equation*}
for $k\geqslant 1$ and $1\leqslant
j\leqslant 2^k$. 
Let $d=dim(M)$. We  choose  $\{e_i, 1\leqslant i\leqslant d\}$, a family of measurable vector
fields  on $M$ with  $\{e_i(z);
1\leqslant i \leqslant d\}$ an orthonomal basis on $T_z M$ for
every $z \in M$. These give rise to an orthonormal basis of $\H_{\omega}^0$:
\begin{equation*}
Z_s^{k,j,i}(\omega)=h_{k,j}(s)\parals_{1/2,s}(\omega)e_i(\omega(1/2)),\
s\in [0,1], \ k\geqslant 0,\ 1\leqslant j \leqslant 2^k,\
1\leqslant i\leqslant d.
\end{equation*}
For each $F\in\F C_b^\infty(\Omega_{a,b})$, let
\begin{equation*}
\Lambda_{k,j}(F)(\omega)=\sum_{i=1}^d[dF(Z^{k,j,i})]^2=\sum_{i=1}^d(\D _0F,Z^{k,j,i})^2_{\H _{\omega}^0}.
\end{equation*}
Then we have
\begin{equation*}
|\D _0F(\omega)|^2_{\H _{\omega}^0}=\sum_{k=0}^{\infty}\sum_{j=1}^{2^k}
\Lambda_{k,j}(F)(\omega),\quad\omega \in \Omega_{a,b}.
\end{equation*}
By \cite[lemma 4.3]{EB1}, there exist
constants $R_2>0$, such that for each $r\in (0,R_2)$, there is a  $\widetilde{T}(r)>0$,
when $T<\widetilde{T}(r)$, for each smooth cylinder function $F$ 
and $\omega \in \Omega_{a,b}$, we have,  
\begin{equation*}
\vGamma   ^T(F)(\omega)I_{\{\omega(1/2) \in U_{a,b}^r \}}
\leqslant C(r)\Lambda_{0,1}(F)(\omega)+
\sum_{l=0}^{\infty}C(r)(T+2^{-l})\sum_{n=1}^{2^l}\Lambda_{l,n}(F)(\omega).
\end{equation*}
Thus, we obtain
\begin{equation}\label{e0.2}
\begin{split}
&\vGamma   ^{T/2^k,j}(F^{[k]})(\omega_1,\dots\omega_{2^k})
I_{\{\omega_j(1/2) \in U_{x_{j-1},x_j}^r \}}\\
&\leqslant C(r)\Lambda_{0,1}^{k,j}(F^{[k]})(\omega_1,\dots\omega_{2^k})
+\sum_{l=0}^{\infty}C(r)(T/2^k+2^{-l})\sum_{n=1}^{2^l}\Lambda_{l,n}^{k,j}
(F^{[k]})(\omega_1,\dots\omega_{2^k}).
\end{split}
\end{equation}
Here $\Lambda_{l,n}^{k,j} (F^{[k]})$ means the corresponding
operation $\Lambda_{l,n}$ is taken with respect to the $j$-th subpath for
function $F^{[k]}$. Since
\begin{equation}
\begin{split}\label{c1.11}
\Lambda_{l,n}^{k,j} (F^{[k]})(\omega_1,\dots\omega_{2^k})&=
\sum_{i=1}^d[dF^{[k]}(Z^{l,n,i})]^2(\omega_1,\dots,\omega_{j-1},\bullet,\omega_{j+1},\dots\omega_{2^k})\\
&=\sum_{i=1}^d 2^k [dF(Z^{l+k,(j-1)2^l+n,i})]^2(\omega)
=2^k\Lambda_{l+k,(j-1)2^l+n}(F)(\omega),
\end{split}
\end{equation}
the second equality above is due to the defintion of $Z^{l,n,i}$ and
some time rescaling procedure. Then by (\ref{e0.2}) and (\ref{c1.11}) we obtain,  

\begin{equation}\label{c.12}
\begin{split}
&\sum_{k=0}^{m-1}q(T/2^k,r)\sum_{j=1}^{2^k}\vGamma   ^{T/2^k,j}(F^{[k]})(\omega_1,\dots\omega_{2^k})
I_{\{\omega_j(1/2) \in U_{x_{j-1},x_j}^r \}}\\
&\leqslant\sum_{k=0}^{m-1}\big\{C(r)q(T/2^k,r)2^k+\sum_{l=0}^kC(r)(2^{k-2l}+T)
q(T/2^{k-l},r)\big\}\sum_{j=1}^{2^k}\Lambda_{k,j}(F)(\omega).
\end{split}
\end{equation}
Let $g(k,T,r):=C(r)q(T/2^k,r)2^k+\sum_{l=0}^kC(r)(2^{k-2l}+T)
q(T/2^{k-l},r)$, by the estimate of $q(T,r)$ in lemma 2.4, we can find
a $T(r)<\widetilde{T}(r)$, such that for each $T< T(r)$, 
$\sup_{k \in \mathbb{N}}g(k,T,r)\leqslant TC(r)$, where $C(r)$ is a constant independant of $T$ and $k$.

So by (\ref{c.12}), for $T<T(r)$, we have,

\begin{equation*}
\begin{split}
&\sum_{k=0}^{m-1}q(T/2^k,r)\cdot\Big(\sum_{j=1}^{2^k}
\int_{U_j}\bigg\{\E  _{x_{j-1},x_j}^{T/2^k,j}
\big[\vGamma   ^{T/2^k,j}(F^{[k]})(\omega_1,\dots \omega_{2^k})
\I_{\{\omega_j(1/2) \in U_{x_{j-1},x_j}^r\}} \big]\\
&\times\prod_{i\neq j}\p  _{x_{i-1},x_i}^{T/2^k}(d\omega_i)\bigg\}
\mu_{a,b}^{2^k,T}(dx)\Big)\\
&\leqslant\sum_{k=0}^{m-1}g(k,T,r)\sum_{j=1}^{2^k}\E  _{a,b}^T[\Lambda_{k,j}(F)]\\
&\leqslant TC(r)\E  _{a,b}^T|\D _0F|_{\H _{\omega}^0}^2
\end{split}
\end{equation*}
\end{proof}

\section{An estimate over discriticized loop spaces}
For each $r\in \mathbb{R}^{+}$ and integer $N\geqslant 1$, define the subset
$U_{a,b}^{r,N}$ of $M^{N-1}$ as,  
\begin{equation}
U_{a,b}^{r,N}:=\big\{(x_1,\dots x_{N-1})\in M^{N-1};\
d(x_{i-1},x_{i})<r,\ 1\leqslant i \leqslant N,\ x_0=a,\ x_{N}=b\big\}.
\end{equation} And recall that 
$\mu_{a,b}^{N,T}$ is the probability measure on $M^{N-1}$ defined 
in (\ref{d1.2}), which is also the joint distribution 
of $\big(\omega(i/N), \ i=1,2,\dots N-1\big)$ under $\p  _{a,b}^{T}$.
We have the following weak  estimates of the variance with respect to $\mu_{a,b}^{N,1}$.
\begin{prp}\label{p2.1}
Let M be a compact simply connected manifold with strict positive Ricci curvature. For any $\eta>0$ small enough, 
$\ 0<r<R_0$, there exists an integer $N_1(\eta,r)>0$, such that for 
sufficiently big integer $l$ ($l>N_1(\eta,r)$ for some constant 
$N_1(\eta,r)$ which only depends on $\eta$ and $r$), there exists an integer $N(l,\eta,r)$, if $N>N(l,\eta,r)$ 
and $f\in C^\infty(M^{N-1})$ with $\supp(f)\subset \overline{U}_{a,a}^{r,N}$, then we have, 
\begin{equation*}
\begin{split}
&\Var(f;\mu_{a,a}^{N,1})\\
&\leqslant
C(l,r)N^{C(l,r)}\e^{2N\eta r^2}
\sum_{i=1}^{N-1}\int_{M^{N-1}}|\nabla_i
f|^2
\d\mu_{a,a}^{N,1}
+C(l,\eta,r)N^{C(l,r)}
\e^{-\frac{N(1-8\eta)r^2}{2}}||f||_{\infty}^2.
\end{split}
\end{equation*}
\end{prp}

\begin{proof} 
First choose an integer $N>>l>>1$, for each $f\in C^\infty(M^{N-1})$ , we define a function 
$f_l: M^{N-l}\mapsto \mathbb{R}$
as following,  
\begin{equation*}
f_l(x_1, \dots, x_{N-l})=\int_{M^{l-1}}f(x_1,\dots,x_{N-l}, y_{1},\dots, y_{l-1})\mu_{x_{N-l},a}^{l,\frac{l}{N}}
(dy_{1}\dots dy_{l-1})
\end{equation*}

 We also introduce a probability measure $\mu_{a,b}^{N,l,T}$ on $M^{N-l}$ as,  
\begin{equation*}
\mu_{a,b}^{N,l,T}(dx_1, \dots, dx_{N-l}):=
\frac{p_{\frac{lT}{N}}(b,x_{N-l})p_{\frac{T}{N}}(x_{N-l},x_{N-l-1})\cdots
p_{\frac{T}{N}}(x_1,a)}{p_T(a,a)}dx_{N-l}\dots dx_1 
\end{equation*}

Let $\wp_l:=\sigma\{\omega(i/N),\ 1\leqslant i \leqslant N-l\}$ be an 
$\sigma$-algebra on $\Omega_{a,a}$ and define a smooth cylinder function 
$\widehat{F}:\Omega_{a,a}\mapsto \mathbb{R}$ as,  
$$\widehat{F}(\omega):=f(\omega(1/N),\dots,\omega(1-1/N)), $$ 
For each $x_i \in M, 1\leqslant i \leqslant N-l$ and
$\omega \in \Omega_{x_{N-l},a}$, let  
$$\widetilde{F}_l(x_1,\dots x_{N-l},\omega):=f(x_1,\dots x_{N-l},\omega(1/l),\dots,\omega(1-1/l)). $$
It is not difficult to check, 
$$\E  _{a,a}\big[\widehat{F}|\omega(1/N)=x_1,\dots,\omega(1-l/N)=x_{N-l}\big]
=f_{l}(x_1,\dots x_{N-l}).$$
and 
\begin{equation*}
\E  _{x_{N-l},a}^{\frac{l}{N}}
\big[\widetilde{F}_l(x_1,\dots x_{N-l},\bullet)\big]=f_{l}(x_1,\dots x_{N-l}) 
\end{equation*}
 
Hence we can obtain,
\begin{equation}
\begin{split}\label{c2.5}
&\Var(f; \mu_{a,a}^{N,1})=\Var(\widehat{F};\p  _{a,a})\\
&=\E  _{a,a}[(\widehat{F}-\E  _{a,a}[\widehat{F}|\wp_{l}])^2]
+\E  _{a,a}[(\E  _{a,a}[\widehat{F}|\wp_{l}]-
\E  _{a,a}[\widehat{F}])^2]\\
&=\int_{M^{N-l}}\Var(\widetilde{F}_l;\p  _{x_{N-l},a}^{\frac{l}{N}})
\mu_{a,a}^{N,l,1}(dx)+\Var(f_{l},\mu_{a,a}^{N,l,1})\\
\end{split}
\end{equation}





Now we are going to estimate $\Var(f_{l},\mu_{a,a}^{N,l,1})$. 
Let $\p  _a^{1}$ be the distribution of a
standard Brownnian motion on compact manifold $M$ starting from $a$ with time parameter $1$, 
which is  a probability measure on the path space $\Omega_{a}$ over $M$ with
starting point $a$ and time $1$.  
Let 
$$\gamma_{a}^{N,l,1}(dx_1,\dots dx_{N-l}):=p_{\frac{1}{N}}(a,x_1),\dots
p_{\frac{1}{N}}(x_{N-l-1},x_{N-l})
\d x_1,\dots \d x_{N-l}$$ 
be a probability
measure on $M^{N-l}$, which is the joint distribution of 
$\big(\omega(i/N),  \omega \in \Omega_{a}\ i=1,2,\dots N-l\big)$ under $\p  _{a}^{1}$.  
By the Poincar\'{e} inequality for $\p  _{a}^{1}$ on the path space
over compact manifold, which is proved in \cite{Fa}, we get, 
\begin{equation*}
\begin{split}
\Var(f_l;\gamma_{a}^{N,l,1})&=\Var
(\overline{F}_l; \p  _a^1)
\leqslant C\E  _a^1|\D \overline{F}_l|_{\H_{\omega} }^2
\leqslant CN\sum_{i=1}^{N-l}\int_{M^{l-1}}|\nabla_i f_l|^2
d\gamma_{a}^{N,l,1}
\end{split}
\end{equation*}
where $\overline{F}_l(\omega):=f_l(\omega(1/N),\dots \omega(1-l/N))$ for 
each $\omega\in\Omega_{a}$, 
and 
$\D $ is the gradient operator related to Bismut tangent norm $|.|_{\H_{\omega}}$ in
path space over compact manifold $M$, and we also use
the relation 
$$|\D \overline{F}_l(\omega)|_{\H_{\omega} }^2\leqslant 
(N-l)\sum_{i=1}^{N-l}|\nabla_i f_l(\omega(1/N),\dots \omega(1-l/N))|^2,\quad \omega\in\Omega_{a},$$
in above inequality which can be checked by direct computation. 

Thus, we have
\begin{equation}
\begin{split}\label{c2.6}
\Var(f_l;\mu_{a,a}^{N,l,1})&=\Var\bigg(f_l;
\frac{p_{\frac{l}{N}}(a,x_{N-l})}{p_{1}(a,a)}\gamma_{a}^{N,l,1}\bigg)\\
&\leqslant C\text{osc}(p_{\frac{l}{N}}(a,\cdot))
N\sum_{i=1}^{N-l}\int_{M^{l-1}}|\nabla_i f_l|^2 \d\mu_{a,a}^{N,l,1}.
\end{split}
\end{equation}
here 
$\text{osc}(g(\cdot)):=\frac{\sup_{x \in M}g(x)}
{\inf_{x \in M}g(x)}$ for any function $g$ over $M$. And by (\ref{c1.0}), if 
$\frac{l}{N}<T(\eta,r)$, then 
$$\text{osc}(p_{\frac{l}{N}}(a,\cdot))\leqslant 
\e^{\frac{N}{l}(\eta^2r^2+\frac{D^2}{2})}$$
where $D$ denotes the diameter of the compact manifold $M$.
So by this and (\ref{c2.6}), if $\frac{l}{N}<T(\eta,r)$, then,  
\begin{equation}\label{d3.2}
\Var(f_l;\mu_{a,a}^{N,l,1})\leqslant 
C N\e^{\frac{N}{l}(\eta^2r^2+\frac{D^2}{2})} 
\sum_{i=1}^{N-l}\int_{M^{l-1}}|\nabla_i f_l|^2 \d\mu_{a,a}^{N,l,1}.
\end{equation}

Now we are going to estimate $|\nabla_i f_l|$, 
it is not hard to see for $1\leqslant i \leqslant N-l-1$,
\begin{equation}\label{d3.3}
|\nabla_i f_l|^2(x_1,\dots x_{N-l})\leqslant 
\int_{M^{l-1}}|\nabla_i f|^2(x_1,\dots x_{N-l},y_1,\dots y_{l-1})
\mu_{x_{N-l},a}^{l,\frac{l}{N}}(dy)
\end{equation}
as for $i=N-l$,
\begin{equation}\label{d3.4}
\begin{split}
&|\nabla_{N-l}f_l|^2(x_1,\dots x_{N-l})=\sup_{|v|=1}\d_{N-l}f_l(v)\\
&\leqslant \int_{M^{l-1}}|\nabla_{N-l} f|^2(x_1,\dots x_{N-l},y_1,\dots y_{l-1})
\mu_{x_{N-l},a}^{l,\frac{l}{N}}(dy)\\
&+\sup_{|v|=1}\Big|\d_z\Big(\E  _{z,a}^{\frac{l}{N}}
(\widetilde{F}_l)\Big)|_{z=x_{N-l}}(v)\Big|^2.
\end{split}
\end{equation}
where $\widetilde{F}_l(\omega):=f(x_1,\dots x_{N-l},\omega(1/l),\dots \omega(1-1/l)),$ is defined
as before. We can also use lemma 2.3 to estimate the differentiation 
of the expectation with starting point as before, but we can not 
make sure $d(a,x_{N-l})\leqslant r$ here to take the vector stated in lemma 2.2, so we have to choose another vector field
$X^{l,v}(s):=\parals_s(1-ls)^{+}v,\ 0\leqslant s \leqslant 1$. Since
for the anti-development $B_s$ in the definition 
of $\delta^{\frac{l}{N}}X$ in lemma 2.1, $$B_s=\beta_s+\int_0^s
\Big(\parals_u^{-1}\nabla\text{log}p_{\frac{(1-u)l}{N}}(\omega(u),a)\Big)du,\ 0\leqslant s<1$$
for some process $\beta_s$ whose distribution is the 
Brownnian motion with time parameter $\frac{l}{N}$ under 
the probability measure $\p_{x_{N-l},a}^{\frac{l}{N}} $(see \cite{Dr3}), then 
we get, 


\begin{equation*}
\begin{split}
&\Var\Big(\delta^{\frac{l}{N}}X^{l,v};\p  _{x_{N-l},a}^{\frac{l}{N}}\Big)
\leqslant \E  _{x_{N-l},a}^{\frac{l}{N}}\Big(\delta^{\frac{l}{N}}X^{l,v}\Big)^2\\
&\leqslant \E  _{x_{N-l},a}^{\frac{l}{N}}\Big[\int_0^{\frac{1}{l}}
\Big(-Nv+\frac{1}{2}\text{Ric}_{\omega(s)}(\parals_s(1-ls)v)\Big)
\Big(\d\beta_s+\parals_s^{-1}\nabla\text{log}p_{\frac{(1-s)l}{N}}(\omega(s),a)\d s\Big)\Big]^2\\
&\leqslant C(l)N^4,\quad v\in S_{x_{N-l}}M,
\end{split}
\end{equation*}
where in the last step of above inequality we use the estimate 
$|\nabla\text{log}p_{s}(x,a)|\leqslant C\big[\frac{d(x,a)}{s}+\frac{1}{\sqrt{s}}\big]$
for the heat kernel in compact manifold $M$. Also note that 
$X^{l,v}(\frac{i}{l})=0 \ \ 1\leqslant i \leqslant l$, so apply lemma 2.3, we have, 
\begin{equation}
\begin{split}\label{c2.8}
&\sup_{|v|=1}\Big|\d_{N-l}\Big(\E  _{x_{N-l},a}^{\frac{l}{N}}
(\widetilde{F}_l)\Big)(v)\Big|^2\\
&\leqslant
\sup_{|v|=1}\Big\{\Big|\E  _{x_{N-l},a}^{\frac{l}{N}}[d\widetilde{F}_l(X^{l,v})]\Big|+\Big[\Var(\delta^{\frac{l}{N}}
X^{l,v};\p  _{x_{N-l},a}^{\frac{l}{N}})\Big]^{1/2}
\Big[\Var(\widetilde{F}_l;\p  _{x_{N-l},a}^{\frac{l}{N}})\Big]^{1/2}\Big\}^2\\
&\leqslant C(l)N^4\Var\Big(\widetilde{F}_l;\p  _{x_{N-l},a}^{\frac{l}{N}}\Big).
\end{split}
\end{equation}

By (\ref{d3.3}), (\ref{d3.4}) and (\ref{c2.8}), we can derive some 
estimate of $|\nabla_i f_l|^2, \ 1\leqslant i \leqslant N-l$, then from that  
and (\ref{c2.5}), (\ref{d3.2}), we can obtain
the following, 
\begin{equation}
\begin{split}\label{f4}
&\Var(f; \mu_{a,a}^{N,1})\leqslant C(l)N\text{exp}
\big(\frac{N}{l}(\eta^2r^2+\frac{D^2}{2})\big)\sum_{i=1}^{N-l}
\int_{M^{N-1}}|\nabla_i f|^2 \mu_{a,a}^{N,1}(dx)\\ 
&+\Big[1+C(l)N^5\text{exp}
\big(\frac{N}{l}(\eta^2r^2+\frac{D^2}{2})\big)\Big]
\int_{M^{N-l}}\Var(\widetilde{F}_l;\p  _{x_{N-l},a}^{\frac{l}{N}})
\mu_{a,a}^{N,l,1}(dx)
\end{split}
\end{equation}

Note that
\begin{equation}\label{c2.9}
\Var\Big(\widetilde{F}_l;\p  _{x_{N-l},a}^{\frac{l}{N}}\Big)
=\Var\Big(f(x_1,\dots x_{N-l},\bullet,\dots,\bullet);\mu_{x_{N-l},a}^{l,\frac{l}{N}}\Big).
\end{equation}
Let $\overline{\mu}_{x_{N-l},a}^{l,\frac{l}{N}}$ be
normalization of $\mu_{x_{N-l},a}^{l,\frac{l}{N}}$ in the
subset $U_{x_{N-l},a}^{r,l}$ of $M^{l-1}$, i.e.
$$\overline{\mu}_{x_{N-l},a}^{l,\frac{l}{N}}(A)=
\mu_{x_{N-l},a}^{l,\frac{l}{N}}(A)/\mu_{x_{N-l},a}^{l,\frac{l}{N}}(U_{x_{N-l},a}^{r,l}),\quad
A\subseteq U_{x_{N-l},a}^{r,l}.$$  
For each smooth function $g$ with support in 
$\overline{U}_{x_{N-l},a}^{r,l}$, we have,
\begin{equation}\label{d1.6}
\begin{split}
&\Var\Big(g;\mu_{x_{N-l},a}^{l,\frac{l}{N}}\Big)\\
&\leqslant \mu_{x_{N-l},a}^{l,\frac{l}{N}}(U_{x_{N-l},a}^{r,l})
\Var\Big(g;
\overline{\mu}_{x_{N-l},a}^{l,\frac{l}{N}}\Big)
+\frac{\Big(1-\mu_{x_{N-l},a}^{l,\frac{l}{N}}(U_{x_{N-l},a}^{r,l})\Big)}
{\mu_{x_{N-l},a}^{l,\frac{l}{N}}(U_{x_{N-l},a}^{r,l})}
||g||_{\infty}^2.
\end{split}
\end{equation}
By asymptotic property (\ref{c1.0}), when $\frac{l}{N}<T(\eta,r)$, 
it satisfies that, 
\begin{equation}
\begin{split}\label{c2.10}
&1-\mu_{x_{N-l},a}^{l,\frac{l}{N}}(U_{x_{N-l},a}^{r,l})=
\mu_{x_{N-1},a}^{l,\frac{l}{N}}(\exists \ 0\leqslant i \leqslant l-1,\ d(z_i,z_{i+1})>r)\\
&\leqslant
\sum_{i=0}^{l-1}\frac{\int_{\d(z_i,z_{i+1})>r}p_{\frac{1}{N}}(x_{N-l},z_1),\dots
p_{\frac{1}{N}}(z_{l-1},a)\d z_1,\dots \d
z_{l-1}}{p_{\frac{l}{N}}(x_{N-l},a)}\\
& \leqslant
l\cdot\frac{\text{exp}(-\frac{(1-4\eta)Nr^2}{2})}{\text{exp}(-\frac{N}{2l}(\eta^2r^2+D^2))}.
\end{split}
\end{equation}
Hence if we choose a sufficient big $l$ such that $\frac{\eta^2r^2+D^2}{l}<2(1-4\eta)r^2$, 
there is an integer 
$\widetilde{N}(\eta,l,r)$, such that when $N>\widetilde{N}(\eta,l,r)$, then
\begin{equation}\label{f5}
\mu_{x_{N-l},a}^{l,\frac{l}{N}}(U_{x_{N-l},a}^{r,l})>\frac{1}{2} 
\end{equation}
  
Since we assume $\supp(f(x_1,\dots
x_{N-1}))\subset \overline{U}_{a,a}^{r,N}$, then, for fixed $x_1,\dots
x_{N-l}$, we have, $$\supp(f(x_1,\dots,x_{N-l},\bullet,\dots,\bullet))\subset
\overline{U}_{x_{N-l},a}^{r,l},$$
hence by (\ref{d1.6}), (\ref{c2.10}) and (\ref{f5}), 
for each integer $l$ sufficiently big, there exists an integer $\widetilde{N}(\eta,l,r)$, for 
each $N>\widetilde{N}(\eta,l,r)$, we have,  
\begin{equation}
\begin{split}\label{c2.11}
&\Var\Big(f(x_1,\dots
x_{N-l},\bullet,\dots,\bullet);\mu_{x_{N-l},a}^{l,\frac{l}{N}}\Big)\\
&\leqslant
\frac{1}{\lambda(U_{x_{N-l},a}^{r,l};\overline{\mu}_{x_{N-l},a}^{l,\frac{l}{N}})}
\times\sum_{i=1}^{l-1}\int|\nabla_{N-i+1} f|^2(x_1,\dots
x_{N-l},z_1,\dots z_{l-1})\mu_{x_{N-l},a}^{l,\frac{l}{N}}(dz)\\
&+2l\cdot\frac{\text{exp}(-\frac{(1-4\eta)Nr^2}{2})}
{\text{exp}(-\frac{N}{2l}(\eta^2r^2+D^2))}
||f||_{\infty}^2,
\end{split}
\end{equation}
where
\begin{equation}\label{f1}
\lambda(U_{x,a}^{r,l};\overline{\mu}_{x,a}^{l,\frac{l}{N}}):=\inf_{g \in C_0^{\infty}(U_{x,a}^{r,l})}
\frac{\int_{U_{x,a}^{r,l}} |\nabla
g|^2\d\overline{\mu}_{x_{N-l},a}^{l,\frac{l}{N}}}{\Var(g;\overline{\mu}_{x_{N-l},a}^{l,\frac{l}{N}})}.
\end{equation}
Therefore, by (\ref{f4}) and (\ref{c2.11}), we have
for each $l$ big enough, $N>\widetilde{N}(\eta,l,r)$ and $\frac{l}{N}<T(\eta,r)$, 
\begin{equation}
\begin{split}\label{c2.12}
&\Var(f;\mu_{a,a}^{N,1})\\& \leqslant
\Big[\frac{C(l)N^5\text{exp}
\big(\frac{N}{l}(\eta^2r^2+D^2/2)\big)}{\inf_{x\in
M}\lambda(U_{x,a}^{r,l};\overline{\mu}_{x,a}^{l,\frac{l}{N}})}\Big]
\sum_{i=1}^{N-1}\int_{M^{N-1}}|\nabla_i
f|^2
\d\mu_{a,a}^{N,1}\\
&~~~~~~~~
+\Big[
C(l)N^5\exp\bigg(N\bigg(-\frac{(1-4\eta)r^2}{2}+\frac{3\eta^2r^2+2D^2}{l}\bigg)
\bigg)
\Big]||f||_{\infty}^2.
\end{split}
\end{equation}

Finally, by (\ref{c2.12}) and 
the estimate of $\lambda(U_{x,a}^{r,l};\overline{\mu}_{x,a}^{l,\frac{l}{N}})$ derived in the 
below lemma 4.2 which is uniformly for all $x \in M$, for each integer $l$ sufficiently big,  
there exists an integer $N(\eta,l,r)>0$, such that if 
$N>N(\eta,l,r)$, then we have, 
\begin{equation}\label{f6}
\begin{split}
&\Var(f;\mu_{a,a}^{N,1})\\& \leqslant
\bigg[C(l,r)N^{C(l,r)}\exp\bigg(N\bigg(\frac{L(\varepsilon)
+\eta^2r^2+D^2/2}{l}+4D\varepsilon\bigg)
\bigg)\bigg]
\sum_{i=1}^{N-1}\int_{M^{N-1}}|\nabla_i
f|^2
\d\mu_{a,a}^{N,1}\\
&+\bigg[C(l,\eta,r)N^{C(l,r)}
\exp\bigg(N\bigg(-\frac{(1-4\eta)r^2}{2}+\frac{3\eta^2r^2+2D^2}{l}\bigg)
\bigg)\bigg]||f||_{\infty}^2.
\end{split}
\end{equation}
 
Note that all the constans $C$ and $L$ in above inequality do not depend on $N$, and
$L$ does not depend on $l$ and the starting point $a$.
So for any fixed $\eta>0$, $0<r<R_0$, we first choose a $\varepsilon=\frac{\eta r^2}{4D}$ to make
$4D\varepsilon=\eta r^2$, then take a $l$ big enough such that 
$\frac{L(\varepsilon)
+\eta^2r^2+D^2/2}{l}<\eta r^2$ and $\frac{3\eta^2r^2+2D^2}{l}<\eta r^2$ for 
the choosen $\varepsilon=\frac{\eta r^2}{4D}$ (i.e.
$l>N_0(\eta,r)$ for some constant $N_0(\eta,r)$ which only depends on $\eta$ and $r$). Hence
by (\ref{f6}), there is a constants $N_1(\eta,r)$, such that for each integer $l>N_1(\eta,r)$,  
there exists an integer $N(\eta,l,r)>0$, such that if 
$N>N(\eta,l,r)$, then we have, 
\begin{equation*}
\begin{split}
&\Var(f;\mu_{a,a}^{N,1})\\
&\leqslant
C(l,r)N^{C(l,r)}\e^{2N\eta r^2}
\sum_{i=1}^{N-1}\int_{M^{N-1}}|\nabla_i
f|^2
\d\mu_{a,a}^{N,1}
+C(l,\eta,r)N^{C(l,r)}
\e^{-\frac{N(1-8\eta)r^2}{2}}||f||_{\infty}^2.
\end{split}
\end{equation*}
By now we have completed the
proof.
\end{proof}

\begin{lem}
Let  $M$ be a compact simply connected manifold with strict Ricci 
curvature.  For $x\in M$, $r<R_0$ and $N\in {\mathbf N}$, $\lambda(U_{x,a}^{r,l};\overline{\mu}_{x,a}^{l,\frac{l}{N}})$
as defined in (\ref{f1}), there exists a constant $T(l,r)$, such that 
when $\frac{l}{N}<T(l,r)$, then for each $\varepsilon>0$ small enough, 
\begin{equation*}
\inf_{x,a \in
M}\lambda(U_{x,a}^{r,l};\overline{\mu}_{x,a}^{l,\frac{l}{N}})\geqslant
\frac{C(l,r)}{N^{C(l,r)}}\exp(-\big(\frac{L(\varepsilon)}{l}+4D\varepsilon\big)\cdot N).
\end{equation*}
where the constant $C(l,r)$ only depends on $l$, $r$ and 
the constant $L(\varepsilon)$ only depends on $\varepsilon$, not on $l$.
\end{lem}
\begin{proof} 
Step (a): Following  \cite{EB2} define a measure $\nu_{a,b}^{l,T}$ on $M^{l-1}$ as an approximating measure:
\begin{equation*}
\nu_{a,b}^{l,T}(\d z)=\text{exp}(-E_{a,b}^l/T)dz_1,\dots dz_{l-1},
\end{equation*}
where
\begin{equation*}
E_{a,b}^l(z_1,\dots
z_{l-1})=\frac{l}{2}\sum_{i=0}^{l-1}d(z_i,z_{i+1})^2,\quad
z_0=a,z_l=b.
\end{equation*}
Let $\overline{\nu}_{a,b}^{l,T}(dz)$ be normalization of
${\nu}_{a,b}^{l,T}(dz)$ in the subset $U_{a,b}^{r,l}$ of
$M^{l-1}$. From \cite[lemma 3.2]{EB2}, for each fixed $l>0$,
$$\varlimsup_{T\downarrow 0}\sup_{a,b \in M}\sup_{U_{a,b}^{r,l}}
\text{osc}\big(\d{\mu}_{a,b}^{l,T}/\d{\nu}_{a,b}^{l,T}\big)\leqslant
C(l,r),$$ So, there is a $T(l,r)>0$ such that for any
$\frac{l}{N}<T(l,r)$,
\begin{equation}\label{c2.13}
\lambda(U_{x,a}^{r,l};\overline{\mu}_{x,a}^{l,\frac{l}{N}})\geqslant
\frac{1}{2C(l,r)}\lambda(U_{x,a}^{r,l};\overline{\nu}_{x,a}^{l,\frac{l}{N}}).
\end{equation}

As in \cite{EB2}, let $U_{a,b,\triangle}^{r,l}:=\overline{U}_{a,b}^{r,l}/\sim$
be the one point compactification of $U_{a,b}^{r,l}$, which is obtained by identifying 
the boundry $\partial U_{a,b}^{r,l}$ as a single point $\triangle$. 
And let $\widetilde{C}([0,1];\overline{U}_{a,b}^{r,l})$ denote the path in 
$\overline{U}_{a,b}^{r,l}$ which is restricted to a continuous path on the space
$U_{a,b,\triangle}^{r,l}$. Then define


\begin{equation}\label{d3.6}
M_{a,b}^{r,l}(z):=\inf_{p \in \mathbf I_{a,b}^{r,l}}
\sup_{s\in[0,1]}E_{a,b}^l(p(s))\quad a,b \in M,
\end{equation}
where $\mathbf I_{a,b}^{r,l}=\big\{p
\in \widetilde{C}([0,1];\overline{U}_{a,b}^{r,l});\
p(0)=z,p(1)=z_0\big\}$ and 
$z_0$ is a minimum point of $E_{a,b}^l$ in
$\overline{U}_{a,b}^{r,l}$. And define 
\begin{equation}\label{d3.7}
m_{a,b}^{r,l}:=\sup_{\overline{U}_{a,b}^{r,l}}(M_{a,b}^{r,l}-E_{a,b}^l)
\end{equation}
In fact, if we take the supremum only among the local minimum points of $E_{a,b}^l$  on 
$\overline{U}_{a,b}^{r,l}$ in the above
definition, the value of $m_{a,b}^{r,l}$ will not change, see lemma 2.1 in \cite{EB2}. 

According to the proof of Theorem 2.2 in \cite{EB2},  for each 
$x,a \in M$, if $\frac{l}{N}$ is less
than some $T(x,a,l)$,  
\begin{equation*}
\lambda(U_{x,a}^{r,l};\overline{\nu}_{x,a}^{l,\frac{l}{N}})\geqslant
C(x,a,l)\Big(\frac{l}{N}\Big)^{3(l-1)d-2}\text{exp}\Big(-\frac{Nm_{x,a}^{r,l}}{l}\Big),\quad
x\in M,
\end{equation*}
where $d$ is the dimension of $M$. 

Now our goal is to confirm that the constants $T(x,a,l)$, $C(x,a,l)$ above can be choosen 
to be independant of $x,a \in M$. From step by step checking the proof 
Theorem 2.2 in \cite{EB2}, if the following three conditions are true, then 
we can find such constants:
\begin{enumerate}
\item
 Uniform estimate on the gradient of the energy function:
there exists a  constant $C(l)>0$ depending only on $l$ such that
$$\sup_{x,a\in M}\sup_{z \in \overline{U}_{x,a}^{r,l}}|\nabla E_{x,a}^l(z)|^2\leqslant C(l).$$

\item
  A lower bound on the size of the tube $U_{x,a}^{r,l}$: 
there exists a  constant $\theta(l)>0 $, such that for all $R<1$,
\begin{equation*}
\sup_{x,a\in M}\sup_{ z\in \partial U_{x,a}^{r,l}} 
{\frac{Vol(B_R(z)/U_{x,a}^{r,l})}{Vol(B_R(z))}} \geqslant \theta(l),
\end{equation*}
where $Vol(A)$ denotes the Riemannian volume of a subset $A$ of $M^{l-1}$.

\item   For $T$ sufficiently small, say smaller than some $T(l)>0$,
there are finite subsets $\Sigma_T^0(x,a)\subset \partial U_{x,a}^{r,l}$
and  $\Sigma_{T}(x,a)\subset \overline{U}_{x,a}^{r,l}$ such that
\begin{itemize}
\item $\Sigma_T^0(x,a)\subset \Sigma_{T}(x,a)$
\item  $\Sigma_{T}(x,a)$ contains a minimum point
$z_0(x,a)$ of $E_{x,a}^l$.
\item $\partial U_{x,a}^{r,l}\subseteq\bigcup_{z\in
\Sigma_{T}^0(x,a)}B_T(z)$,  $\overline{U}_{x,a}^{r,l}\subseteq\bigcup_{z\in
\Sigma_{T}(x,a)}B_T(z)$.
\item 
$\sup_{x,a\in M}  \# \Sigma_T(x,a)\leqslant C(l) T^{-(l-1)d}$ for some constants $C(l)$.

\end{itemize}
where $\#$ means the number of elements in a finite set.

\end{enumerate}

%

Since $R_0$ from proposition 3.2 is less than 
the injective radius of compact manifold $M$,  when $r\in (0,R_0)$, $E_{x,a}^l$ is differentiable in 
the domain $U_{x,a}^{r,l}$ and condition 1 can be checked by direct computation. 
From the proof of Corollary 3.3 in \cite{EB2}, condition 2 is
true. 

For condition 3, note that there is a $T(l)>0$,  
for each $T<T(l)$, due to the compactness of $M^{l-1}$, we can find
a finite subset $\widetilde{\Sigma}_T \subseteq M$ such that 
$ M^{l-1} \subseteq \bigcup_{z\in
\widetilde{\Sigma}_T}B_{T}(z)$ and
$\#\widetilde{\Sigma}_T \leqslant C(l)T^{-(l-1)d}$. Now since
$M^{l-1}\subseteq\bigcup_{z\in \widetilde{\Sigma}_{T/2}}B_{T/2}(z)
$, we start to construct the set $\Sigma_{T}(x,a)$ as
following:
\begin{enumerate}
\item [(i)] if $z \in \widetilde{\Sigma}_{T/2}$ and
$B_{T/2}(z)\subset U_{x,a}^{r,l}$, then add such $z$
into $\Sigma_{T}(x,a)$;

\item[(ii)] if $z \in \widetilde{\Sigma}_{T/2}$ and
$B_{T/2}(z)\cap \partial U_{x,a}^{r,l}\neq \varnothing$,
then take a point $\tilde{z} \in B_{T/2}(z) \cap \partial
U_{x,a}^{r,l}$ and add this point $\tilde{z}$ into
$\Sigma_{T}(x,a)$.

\item[(iii)] add a  minimum point $z_0(x,a)$ of $E_{x,a}^l$ on
$\overline{U}_{x,a}^{r,l}$ into $\Sigma_{T}(x,a)$.

\end{enumerate}
Since in (ii), $B_T(\tilde{z})\supseteq B_{T/2}(z)$,
we have
\begin{equation*}\begin{split}&~~~~~~~~\bigcup_{\tilde{z}\in \Sigma_{T}(x,a)}B_T(\tilde{z})\supseteq
\bigcup_{z \in
\widetilde{\Sigma}_{T/2}}B_{T/2}(z)\supseteq
M^{l-1}\supseteq\overline{U}_{x,a}^{r,l}
\\&\bigcup_{\tilde{z}\in \Sigma_{T}(x,a)\cap \partial U_{x,a}^{r,l} }B_T(\tilde{z})\supseteq
\bigcup_{z \in \widetilde{\Sigma}_{T/2};\
B_{T/2}(z)\cap \partial U_{x,a}^{r,l}\neq
\varnothing}B_{T/2}(z)\supseteq \partial
U_{x,a}^{r,l},\quad x,a\in M\end{split}\end{equation*} and
$\sharp\Sigma_{T}(x,a)\leqslant\sharp
\widetilde{\Sigma}_{T/2}+1 \leqslant
2^{(l-1)d}C(l)T^{-(l-1)d}$, so condition 3 are satisfied.

 By the above argument, we can find constants $T(l)$ and $C(l)$,  which are most importantly independent of $x$ and $a$ ,
such that  if 
$\frac{l}{N}<T(l)$, then

\begin{equation}\label{c2.14}
 { \lambda(U_{x,a}^{r,l};\overline{\nu}_{x,a}^{l,\frac{l}{N}})}
  \geqslant
C(l)\Big(\frac{l}{N}\Big)^{3(l-1)d-2}\text{exp}\Big(-\frac{Nm_{x,a}^{r,l}}{l}\Big) .\end{equation}

Step (b): In the following, we try to give some uniform estimate about
$m_{x,a}^{r,l}$. 
As in \cite{EB2}, define the energy of a path
$\gamma \in \Omega_{a,b}$(possibly infinite) as:
\begin{equation*}
E(\gamma):=\frac{1}{2}\sup\sum_{i=0}^{k-1}
\frac{d(\gamma(s_i),\gamma(s_{i+1}))^2}{s_{i+1}-s_i}
\end{equation*}
where the supremum is obtained over all partitions
$0=s_0<s_1<\dots s_k=1$. Assume $a,b \in M$ and $a$ is not conjugate
to $b$, let $\Xi_{a,b}$ denote the set of all geodesics (i.e.
critical points of $E$) in $\Omega_{a,b}$, and let
$\Xi_{a,b}^{\text{min}}$ denote the subset of all local energy
minimum. Fix a global energy minimum geodesic $\gamma_{a,b} \in
\Omega_{a,b}$, then for each  geodesic $\gamma
\in \Xi_{a,b}$, we define:
\begin{equation*}
M_{a,b}(\gamma):=\inf_{H \in \mathbf I }\sup_{s \in [0,1]}E\circ H(s)
\end{equation*}
where $\mathbf I=\big\{H \in
C([0,1],\Omega_{a,b}); H(0)=\gamma, H(1)=\gamma_{a,b} \big\}$.
And define
\begin{equation*}
m_{a,b}:=\sup\big\{M_{a,b}(\gamma)-E(\gamma);\ \gamma\in
\Xi_{a,b}^{\text{min}}\big\}.
\end{equation*}

The item $m_{a,b}$ can be viewed as an infinite dimensional 
version of the item (\ref{d3.7}). Futhermore, every point $z\in U_{a,b}^{r,l}$ corresponds
to a piecewise geodesic in $M$, so intuitively we may have more choices to take supremum in 
defining $M_{a,b}$ than in defining $M_{a,b}^{r,l}$ as 
(\ref{d3.6}). In fact, according to the proof of
Corollary 1.5 in \cite{EB2}, we have, 
\begin{equation}\label{d5.2}
m_{a,b}^{r.l}\leqslant m_{a,b},\quad r\in(0,\text{inj} M),\ \ 
l\in \mathbb{N}^+.
\end{equation}

For $0<r<R_0$, choose a $\varepsilon>0$,  satisfying with
$r+\varepsilon<\text{inj}M$. For any $x \in M$, $a \in M$ and $\tilde{x} \in
B_{\varepsilon}(x)$, $\tilde{a} \in
B_{\varepsilon}(a)$, if $z=(z_1,\dots z_{l-1})\in
U_{\tilde{x},\tilde{a}}^{r,l}$, then 
\begin{equation*}
\begin{split}
&d(z_1,x)\leqslant d(x,\tilde{x})+d(z_1,\tilde{x})< r+\varepsilon \
\ \ \ d(z_{l-1},a)\leqslant d(a,\tilde{a})+d(z_{l-1},\tilde{a})< r+\varepsilon \\
&\ \text{and}\ \ d(z_i,z_{i+1})<r, \ 1\leqslant i \leqslant l-2
\end{split}
\end{equation*}
which means $z \in U_{x,a}^{r+\varepsilon,l}$, hence we have
$\overline{U}_{\tilde{x},\tilde{a}}^{r,l}\subseteq
U_{x,a}^{r+\varepsilon,l}$. 

Suppose $z_0(\tilde{x},\tilde{a})$ be a minimum point of 
$E_{\tilde{x},\tilde{a}}^{l}$ on $\overline{U}_{\tilde{x},\tilde{a}}^{r,l}$
, and $z_0(x,a)$ be a minimum point of $E_{x,a}^{l}$ on
$\overline{U}_{x,a}^{r+\varepsilon,l}$, by the definition of $M_{a,b}^{r,l}$ 
in (\ref{d3.6}), for each
$\delta>0$ and each $z \in \overline{U}_{\tilde{x},\tilde{a}}^{r,l}\subseteq
U_{x,a}^{r+\varepsilon,l}$, there exists a path $q_1 \in
\widetilde{C}([0,1];\overline{U}_{x,a}^{r+\varepsilon,l})$, such
that $q_1(0)=z, q_1(1)=z_0(x,a)$, and
\begin{equation}\label{d3.8}
E_{x,a}^l\circ q_1(s)\leqslant
E_{x,a}^l(z)+m_{x,a}^{r+\varepsilon,l}+\delta ,\ 0\leqslant s \leqslant 1
\end{equation}
As the same reason, we can find a a path $q_2 \in
\widetilde{C}([0,1];\overline{U}_{x,a}^{r+\varepsilon,l})$
with $q_2(0)=z_0(\tilde{x},\tilde{a}), q_2(1)=z_0(x,a)$ and 
\begin{equation}\label{d3.9}
E_{x,a}^l\circ q_2(s)\leqslant
E_{x,a}^l(z_0(\tilde{x},\tilde{a}))+
m_{x,a}^{r+\varepsilon,l}+\delta,\ \ 
0\leqslant s \leqslant 1. 
\end{equation}

Let
\begin{equation*}
q(s)=
\begin{cases}
q_1(2s)\ &\text{if}\ 0<s\leqslant \frac{1}{2},\\
q_2(2-2s)\ &\text{if}\  \frac{1}{2}<s\leqslant 1
\end{cases}
\end{equation*} and $\tau=\inf\{s; q(s) \in \partial
U_{\tilde{x},\tilde{a}}^{r,l}\}\wedge1$, $\hat{\tau}=\sup\{s; q(s) \in
\partial U_{\tilde{x},\tilde{a}}^{r,l}\}\vee1$. Define
\begin{equation*}
\widetilde{q}(s)=
\begin{cases}
q(s)\ &\text{if}\ s\in [0,\tau)\cup (\hat{\tau},1],\\
q(\hat{\tau})\ &\text{if}\  s\in [\tau,\hat{\tau}].
\end{cases}
\end{equation*}
Then $\tilde{q} \in \widetilde{C}([0,1];
\overline{U}_{\tilde{x},\tilde{a}}^{r,l})$ and $\widetilde{q}(0)=z,\
\widetilde{q}(1)=z_0(\tilde{x},\tilde{a})$. Note that for each $z \in
\overline{U}_{\tilde{x},\tilde{a}}^{r,l}$,
\begin{equation}\label{d3.10}
\begin{split}
&|E_{\tilde{x},\tilde{a}}^l(z)-E_{x,a}^l(z)|\\
&=\big|\frac{l(d(z_1,x)^2-d(z_1,\tilde{x})^2)}{2}
+\frac{l(d(z_{l-1},a)^2-d(z_{l-1},\tilde{a})^2)}{2}\big|\\
&\leqslant (d(a,\tilde{a})+ d(x, \tilde{x}))D l\leqslant 2l D\varepsilon 
\end{split}
\end{equation}
where $D$ is the diameter of the manifold $M$. Then, by
(\ref{d3.8}), (\ref{d3.9}), (\ref{d3.10}) and the definition of 
$\tilde{q}$, we have
\begin{equation*}
\begin{split}
E_{\tilde{x},\tilde{a}}^l\circ\widetilde{q}(s)&\leqslant
E_{x,a}^l\circ\widetilde{q}(s)+2lD\varepsilon\leqslant
\text{max}\{E_{x,a}^l(z),E_{x,a}^l(z_0(\tilde{x},\tilde{a}))\}
+ m_{x,a}^{r+\varepsilon,l}+\delta+ 2lD\varepsilon\\
&\leqslant
\text{max}\{E_{\tilde{x},\tilde{a}}^l(z),E_{\tilde{x},\tilde{a}}^l(z_0(\tilde{x},\tilde{a}))\}
+ m_{x,a}^{r+\varepsilon,l}+\delta+ 4lD\varepsilon\\
&=E_{\tilde{x},\tilde{a}}^l(z)+  m_{x,a}^{r+\varepsilon,l}+\delta+
4lD\varepsilon,\quad 0\leqslant s\leqslant 1.
\end{split}
\end{equation*}
The equality in the last step above is due to the 
fact that $z_0(\tilde{x},\tilde{a})$ is a minimum point of 
$E_{\tilde{x},\tilde{a}}^l$ on $\overline{U}_{\tilde{x},\tilde{a}}^{r,l}$. 
 Thus, according to the above inequality and the definition 
of $M_{\tilde{x},\tilde{a}}^{r,l}$, and by the arbitrary of
$\delta$, we obtain $M_{\tilde{x},\tilde{a}}^{r,l}(z)\leqslant E_{\tilde{x},\tilde{a}}^l(z)+
m_{x,a}^{r+\varepsilon,l}+ 4lD\varepsilon$. Hence, by this
(\ref{d5.2}) and the definition of $m_{\tilde{x},\tilde{a}}^{r,l}$, when $d(x,\tilde{x})<\varepsilon$ and $d(a,\tilde{a})<\varepsilon$, we have
\begin{equation}\label{c2.15}
m_{\tilde{x},\tilde{a}}^{r,l}\leqslant
m_{x,a}^{r+\varepsilon,l}+4lD\varepsilon \leqslant m_{x,a}+
4lD\varepsilon.
\end{equation}
By \cite[Theorem 1.4]{EB2}, when $M$ is a compact simply connected manifold
with strict Ricci curvature, we have $m_{a,b}<\infty$ for each pair of 
$a,b\in M$ if $a$ is not conjugate to $b$.
Since for any $\varepsilon>0,\ a \in M$, there exists a finite set
$\Theta_{\varepsilon,a}\subseteq\{x\in M:x~\hbox{is not conjugate to}~a\}$
such that $\bigcup_{x \in
\Theta_{\varepsilon,a}}B_{\varepsilon}(x)\supseteq M$, then by
(\ref{c2.15}), for each $a, b \in M$ with $d(a,b)<\varepsilon$,  
\begin{equation}\label{c2.16}
 \sup_{y\in M}m_{y,b}^{r,l}\leqslant \sup_{x
\in \Theta_{\varepsilon,a}}m_{x,a}+4lD\varepsilon.
\end{equation}
As the same way, there is a finite set $\Theta_{\varepsilon}$, such that
$\bigcup_{x \in
\Theta_{\varepsilon}}B_{\varepsilon}(x)\supseteq M$, by (\ref{c2.15})
and (\ref{c2.16}),
\begin{equation}\label{c2.161}
 \sup_{y,b \in M}m_{y,b}^{r,l}\leqslant \sup_{a \in \Theta_{\varepsilon}}\sup_{x
\in \Theta_{\varepsilon,a}}m_{x,a}+4lD\varepsilon.
\end{equation} 
Let 
\begin{equation*}
L(\varepsilon):=\sup_{a \in \Theta_{\varepsilon}}\sup_{x
\in \Theta_{\varepsilon,a}}m_{x,a}<+\infty.
\end{equation*}
So, by (\ref{c2.13}), (\ref{c2.14}) and (\ref{c2.161}), if 
$\frac{l}{N}$ less than some $T(l,r)$, then
\begin{equation}\label{c2.17}
\inf_{x,a \in
M}\lambda(U_{x,a}^{r,l};\overline{\mu}_{x,a}^{l,\frac{l}{N}})\geqslant
\frac{C(l,r)}{N^{C(l,r)}}\text{exp}\Big(-\big(\frac{L(\varepsilon)}{l}+4D\varepsilon\big)\cdot N\Big).
\end{equation}
where constant $C(l,r)$ only depends on $l$, $r$, 
by now we have completed the proof.
\end{proof}

\section{The Main Theorem}
\begin{thm}\label{MainTheorem} Let $M$ be a  simply connected compact manifold  with strict
positive Ricci curvature.  For any small $\alpha>0$, there exists a constant $s_0>0$ such
that the following weak Poincar\'{e}  inequality holds, i.e.
\begin{equation}\label{c3.11}
\Var(F;\p  _{a,a})\leqslant
\frac{1}{s^\alpha}\EE_{a,a}(F,F) +
s||F||_{\infty}^2,\quad s\in(0,s_0),\ \ F\in \scr{D}(\scr{E}_{a,a}).
\end{equation}
The constants $s_0$ does not depend on the starting point $a \in M$. 
\end{thm}
\begin{proof}
It suffices to show that (\ref{c3.11}) holds for $F\in\F
C_b^\infty(\Omega_{a,a})$. Let
$\omega_i(s):=\omega(\frac{i-1+s}{N})$ for each 
$\omega \in \Omega_{a,a}$. For a function $F\in\F
C_b^\infty(\Omega_{a,a})$, as in the proof of proposition 3.2, 
there is a unique function $F^{[N]}$ defined on
$\bigcup_{(x_1,\dots, x_{N-1})\in
M^{N-1}}\Pi_{i=1}^N\Omega_{x_{i-1},x_i}$ such that,
\begin{equation*}
F^{[N]}(\omega_1,\omega_2,\dots,\omega_N)=F(\omega),\quad
\omega\in\Omega_{a,a},
\end{equation*}
Step (a): We first assume $F(\omega)=0$ if $\omega \in \Omega_{a,a}$ and  
$\big(\omega(1/N),\omega(2/N),\dots\omega(1-1/N)\big)$ is not in
$U_{a,a}^{r,N}$ for a fixed $N>N(\eta,r,l)$ with $l>N_1(\eta,r)$, here 
$N_1(\eta,r)$ and $N(\eta,l,r)$ are the 
constants we  get in proposition \ref{p2.1}. Let
\begin{equation*}
\begin{split}
&f^{[N]}(x_1,x_2,\dots x_{N-1}):= \int F^{[N]} \prod_{i=1}^N 
\p  _{x_{i-1},x_i}^{\frac{1}{N}}(d\omega_i)\\
&=\E  _{a,a}\big[F(\omega)|\omega(1/N)=x_1,\dots
\omega(1-1/N)=x_{N-1}\big],\quad (x_1,\dots,
x_{N-1})\in M^{N-1}.
\end{split}
\end{equation*}
Let $\Im_N:=\sigma\{\omega(i/N),\ 1\leqslant i \leqslant N-1\}$ be an $\sigma$- algebra on
$\Omega_{a,a}$, then we have,
\begin{equation}
\begin{split}\label{c3.2}
&\Var(F;\p  _{a,a})\\
&=\E  _{a,a}[(F-\E  _{a,a}[F|\Im_{N}])^2]
+\E  _{a,a}[(\E  _{a,a}[F|\Im_{N}]-
\E  _{a,a}[F])^2]\\
&=\int_{M^{N-1}}\Var(F^{[N]};\bigotimes_{i=0}^{N-1}
\p  _{x_{i},x_{i+1}}^{\frac{1}{N}})
d\mu_{a,a}^{N,1}+\Var(f^{[N]},\mu_{a,a}^{N,1})\\
&\leqslant 
\int_{U_{a,a}^{r,N}}\bigg\{\sum_{j=1}^N\int\Var_j(F^{[N]};
\p  _{x_{j-1},x_j}^{\frac{1}{N}})\prod_{i\neq
j} \p  _{x_{i-1},x_i}^{\frac{1}{N}}(d\omega_i)
\bigg\}\mu_{a,a}^{N,1}(dx)+
\Var(f^{[N]};\mu_{a,a}^{N,1})
\end{split} 
\end{equation}

where $\Var_j$ is the variance to the $j$th subpath. Note that  $f^{[N]}$ is smooth with support in $\overline{U}_{a,a}^{r,N}$
and $||f^{[N]}||_{\infty}\leqslant ||F||_{\infty}$, from Proposition \ref{p2.1}, 
if $N>N(\eta,l,r)$, then 
\begin{equation}
\begin{split}\label{c3.3}
&\Var(f^{[N]};d\mu_{a,a}^{N,1})\\& \leqslant
C(l,r)N^{C(l,r)}\e^{2N\eta r^2}
\sum_{i=1}^{N-1}\int_{M^{N-1}}|\nabla_i
f^{[N]}|^2
\d\mu_{a,a}^{N,1}
+C(l,\eta,r)N^{C(l,r)}
\e^{\frac{-N(1-8\eta)r^2}{2}}||F||_{\infty}^2.
\end{split}
\end{equation}
According to the proof of lemma 6.1 and lemma 6.2 in \cite{EB1} (since the support of $f^{[N]}$ is in $\overline{U}_{a,a}^{r,N}$, we can choose some vector with better asymptotic property in 
the estimate of the derivative of  expectation with pinned Wiener measure), there exists a constant $C(r)$, 
such that
\begin{equation}
\begin{split}\label{c3.4}
&\sum_{i=1}^{N-1}\int|\nabla_i f^{[N]}|^2
\d\mu_{a,a}^{N,1}\leqslant C(r)N \E  _{a,a}|\D _0F|_{\H _{\omega}^0}^2\\
&+C(r)N\int_{U_{a,a}^{r,N}}\bigg\{\sum_{j=1}^N
\int_{M^{N-1}}\Var_j(F^{[N]};\p  _{x_{j-1},x_j}^{\frac{1}{N}})\prod_{i\neq
j} \p  _{x_{i-1},x_i}^{\frac{1}{N}}(d\omega_i)
\bigg\}\mu_{a,a}^{N,1}(dx).
\end{split}
\end{equation}
By proposition \ref{l1.1}, if $\frac{1}{N}<T(\eta,r)$, then
\begin{equation}
\begin{split}\label{c3.5}
&\int_{U_{a,a}^{r,N}}\bigg\{\sum_{j=1}^N\int\Var_j
(F^{[N]};\p  _{x_{j-1},x_j}^{\frac{1}{N}})\prod_{i\neq
j} \p  _{x_{i-1},x_i}^{\frac{1}{N}}(d\omega_i)
\bigg\}\mu_{a,a}^{N,1}(dx_1,\dots dx_{N-1})\\
&\leqslant \frac{C(r)}{N}\int\sum_{j=1}^N|\D _{0,(j)}F^{[N]}
(\omega_1,\dots\omega_N)|^2_{\omega_j}\p  _{a,a}(d\omega)
+NC(\eta,r)\text{e}^{-\frac{(1-4\eta)Nr^2}{2}}||F||_{\infty}^2,
\end{split}
\end{equation}
where $\D _{0,(j)}$ means the gradient $\D _0$ of $F^{[N]}$ with
respect to the $j$th subpath. 
According to (6.16) in the proof of lemma 6.3 in \cite{EB1}, we 
have the following relation, 
\begin{equation}\label{d4.1}
\sum_{j=1}^N|\D _{0,(j)}F^{[N]}
(\omega_1,\dots\omega_N)|^2_{\omega_j}\leqslant 
N|\D _0F|_{\H _{\omega}^0}^2,\quad \omega\in \Omega_{a,a}.
\end{equation}

By (\ref{c3.2}), (\ref{c3.3}), (\ref{c3.4}), (\ref{c3.5})
and (\ref{d4.1}) if $N>N(\eta,l,r)$ with $l>N_1(\eta, r)$, then 
\begin{equation}
\begin{split}\label{c3.6}
&\Var(F;\p  _{a,a}^{1})\\
& \leqslant
C(l,r)N^{C(l,r)}\e^{2N\eta r^2}
\E  _{a,a}|\D _0F|_{\H _{\omega}^0}^2
+C(l,\eta,r)N^{C(l,r)}\e^{{\frac{-N(1-8\eta)r^2}{2}}}||F||_{\infty}^2.
\end{split}
\end{equation}

Step (b): Now let's consider general $F\in\F C_b^\infty(\Omega_{a,a})$. Define a 
smooth cut-off function on $\Omega_{a,a}$ as, 
\begin{equation*}
\varPsi_N(\omega):=\prod_{i=1}^N
\varphi\Big(d\Big(\omega\Big(\frac{i-1}{N}\Big)\Big),d\Big(\omega\Big(\frac{i}{N}\Big)\Big)\Big)
\end{equation*}
where $\varphi$ is defined as in the proof of proposition \ref{l1.1}. By
the proof of proposition \ref{l1.1}, if $\frac{1}{N}<T(\eta,r)$, 
\begin{equation}\label{c3.7}
\p  _{a,a}(\varPsi_N\neq 1)\leqslant
N\text{exp}\Big(-\frac{(1-4\eta)Nr^2}{2}\Big),\quad
|\D _0\varPsi_N(\omega)|_{\H _{\omega}^0}\leqslant \frac{6N}{\eta r}.
\end{equation}
Then note that $F\varPsi_N(\omega)=0$ when $\big(\omega(1/N),\omega(2/N),\dots\omega(1-1/N)\big)$ is not in
$U_{a,a}^{r,N}$, hence by
(\ref{c3.6}) and (\ref{c3.7}), if $N>N(\eta,r,l)$ with $l>N_1(\eta, r)$ , we obtain
\begin{equation}
\begin{split}\label{c3.8}
&\Var(F;\p  _{a,a}) \leqslant \Var(F\varPsi_N;\p  _{a,a})+ 3
\p_{a,a}(\varPsi_N\neq 1)||F||_{\infty}^2\\
&\leqslant C(l,r)N^{C(l,r)}\e^{2N\eta r^2}
\E  _{a,a}|\D _0(F\varPsi_N)|_{\H _{\omega}^0}^2
+C(l,\eta,r)N^{C(l,r)}\e^{{\frac{-N(1-8\eta)r^2}{2}}}||F||_{\infty}^2\\
&\leqslant
C(l,r)N^{C(l,r)}\e^{2N\eta r^2}
\E  _{a,a}|\D _0F|_{\H _{\omega}^0}^2
+C(l,\eta,r)N^{C(l,r)}
\e^{-\frac{N(1-8\eta)r^2}{2}}
||F||_{\infty}^2.
\end{split}
\end{equation}
Let $s:=C(l,\eta,r)N^{C(l,r)}\text{e}^{-\frac{(1-8\eta)Nr^2}{2}}$
in (\ref{c3.8}), then 
$s$ tends to zero when $N$ tends to infinity,  in particular, 
for any small $\alpha>0$, we can choose a $\eta$ small enough, so that there is a constant  $s_0(\eta,r,l,\varepsilon,\alpha)$($s_0$ does not depend on the starting point 
$a$ of the loop space), such that, 
\begin{equation*}
\Var(F;\p  _{a,a})\leqslant
\frac{1}{s^{\alpha}}\scr{E}_{a,a}
(F,F) + s||F||_{\infty}^2,\quad
s\in(0,s_0), F\in \scr{D}(\scr{E}_{a,a}).
\end{equation*}
By now we have completed the proof.
\end{proof}

\end{document}